\documentclass[a4paper, preprint, showkeys]{revtex4}

\usepackage[ansinew]{inputenc}
\usepackage{amsmath, amsfonts} 
\usepackage{amsthm, amssymb}
\usepackage[english]{babel}
\usepackage{lmodern}
\usepackage{graphicx}
\usepackage{subfigure}

\newtheorem{proposition}{Proposition}

\theoremstyle{definition}
\newtheorem{definition}{Definition}
\newtheorem{remark}{Remark}

\usepackage{verbatim}
\usepackage{array}
\usepackage{float}
\usepackage{subfigure}
\usepackage{color}

\newcommand{\guillemets}[1]{``#1''}
\newcommand{\ve}[1]{\mathbf{#1}}

\begin{document}

\title{Isostables, isochrons, and Koopman spectrum \\ for the action-angle representation of stable fixed point dynamics}
\author{A. Mauroy}
\email{alex.mauroy@engr.ucsb.edu}
\author{I. Mezic}
\email{mezic@engr.ucsb.edu}
\author{J. Moehlis}
\email{moehlis@engineering.ucsb.edu}

\affiliation{Department of Mechanical Engineering, University of California Santa Barbara, Santa Barbara, CA 93106, USA}

\begin{abstract}

For asymptotically periodic systems, a powerful (phase) reduction of the dynamics is obtained by computing the so-called isochrons, i.e. the sets of points that converge toward the same trajectory on the limit cycle. Motivated by the analysis of excitable systems, a similar reduction has been attempted for non-periodic systems admitting a stable fixed point. In this case, the isochrons can still be defined but they do not capture the asymptotic behavior of the trajectories. Instead, the sets of interest---that we call \guillemets{isostables}---are defined in literature as the sets of points that converge toward the same trajectory on a stable slow manifold of the fixed point. However, it turns out that this definition of the isostables holds only for systems with slow-fast dynamics. Also, efficient methods for computing the isostables are missing.

The present paper provides a general framework for the definition and the computation of the isostables of stable fixed points, which is based on the spectral properties of the so-called Koopman operator. More precisely, the isostables are defined as the level sets of a particular eigenfunction of the Koopman operator. Through this approach, the isostables are unique and well-defined objects related to the asymptotic properties of the system. Also, the framework reveals that the isostables and the isochrons are two different but complementary notions which define a set of action-angle coordinates for the dynamics. In addition, an efficient algorithm for computing the isostables is obtained, which relies on the evaluation of Laplace averages along the trajectories. The method is illustrated with the excitable FitzHugh-Nagumo model and with the Lorenz model. Finally, we discuss how these methods based on the Koopman operator framework relate to the global linearization of the system and to the derivation of special Lyapunov functions.

\end{abstract}


\maketitle

\section{Introduction}

Among the abundant literature on networks of coupled systems, a vast majority of studies focus on asymptotically periodic systems (i.e. coupled oscillators) while only a few consider coupled systems characterized by a stable fixed point. This is particularly surprising since the latter can exhibit excitable regimes that are relevant in many situations (e.g. neuroscience \cite{Izi_book}). One reason for this disproportion is probably related to phase reduction methods. For asymptotically periodic systems, powerful phase reduction methods turn the (complex, high-dimensional) system into a phase oscillator evolving on the circle, making the analysis of complex networks more amenable to mathematical analysis \cite{Brown,Malkin,Winfree}. In contrast, in the case of systems admitting a stable fixed point, the development of equivalent reduction methods is more recent and a general framework is still in its infancy.

The goal of reduction methods is to assign the same value to a (codimension-$1$) set of initial conditions that are characterized by the same asymptotic behavior, in turn designing a coordinate on the state space. In the case of asymptotically periodic systems, these sets of identical (phase) value are the so-called \emph{isochrons}, which approach the same trajectory on the limit cycle \cite{Winfree2}. This concept has been recently extended to heteroclinic cycles \cite{Thomas}. For systems admitting a stable focus, the isochrons (or isochronous sections) can still be defined as the sets of points that are invariant under a particular return map \cite{Gine,Sabatini}. This notion is of particular interest in the case of weak foci (i.e. with purely imaginary eigenvalues) and non-smooth vector fields, where the existence of isochrons is a non-trivial problem related to the stability of the fixed point. However, the isochrons provide in this case no information on the asymptotic convergence of the trajectories toward the fixed point and are not useful for the system reduction. (Note also that they do not exist for fixed points with real eigenvalues.) Therefore, the isochrons must be complemented by another family of sets: the so-called \emph{isostables}.

Excitable systems are characterized by slow-fast dynamics with a stable fixed point and, in the plane, they admit a particular trajectory---the \emph{transient attractor} or slow manifold--- that temporarily attracts all the trajectories as they approach the fixed point. In this case, the isostables are naturally defined as the sets of points that converge to the same trajectory on the transient attractor \cite{Rabinovitch}. (Note that these sets are called \guillemets{isochrons} in \cite{Rabinovitch}, but we feel that the proper sense is \guillemets{isostables} instead, in order to avoid the confusion with the isochrons of foci studied in \cite{Gine,Sabatini}.) For non-planar systems possessing a multi-dimensional slow manifold or center manifold, a (more rigorous) framework was previously developed in \cite{Roberts1,Roberts2}. In that work, the sets of interest (called \guillemets{projection manifolds} in \cite{Roberts1}) are closely related to the notion of isostable and correspond to the invariant fibers of the (slow or center) manifold, i.e. the sets of initial conditions characterized by the same long-term behavior on that manifold. Through the reduction obtained with the isostables, excitable systems have been studied in various contexts (sensitivity to periodic pulses \cite{Coombes,Ichinose,Rabinovitch2}, network synchronization \cite{Masuda}, etc.).


Since the isostables provide a characterization of the system dynamics around the fixed point, their computation is also desirable for systems which do not contain multiple time scales (i.e. with no slow or center manifold). For instance, the computation of the isostables can be useful to achieve an optimal control that minimizes the time of convergence toward a steady state or to investigate the delay of convergence to a stable equilibrium in decision-making models \cite{Trotta}. But in these cases, a more general framework is required, which defines the isostables as particular (and unique) codimension-$1$ sets capturing the asymptotic behavior of the system. In addition, the computation of the isostables through backward integration \cite{Rabinovitch} or normal form of the dynamics \cite{Roberts2} is limited to a neighborhood of the slow manifold. In this context, an efficient method for computing the isostables in the entire basin of attraction is also missing.

In this paper, we propose a general framework for the reduction of systems admitting a stable fixed point, which is not limited to excitable systems with slow-fast dynamics. This approach is based on the spectral properties of the so-called \emph{Koopman operator} \cite{Mezic,MezicBana}. More precisely, we propose a \emph{general} and \emph{unique} definition of the isostables in terms of a particular eigenfunction of the Koopman operator. In addition, the framework yields an efficient method to compute the isostables in the whole basin of attraction. This method relies on the estimation of \emph{Laplace averages} along the trajectories and can be seen as an extension of the approach recently developed in \cite{Mauroy_Mezic} to compute the isochrons of limit cycles.

Viewed through the Koopman operator framework, the isostables and the isochrons appear to be two different but complementary concepts. On the one hand, they are different since they are related to the absolute value and to the argument, respectively, of the eigenfunction of the Koopman operator. On the other hand, they are complementary in the sense that they define a set of \emph{action-angle coordinates} for the system dynamics. This action-angle representation is related to important properties of the isotables, such as the global linearization of the dynamics and the derivation of special Lyapunov functions, that we discuss in the paper.

The paper is organized as follows. In Section \ref{sec_iso}, we introduce the concept of isostable in the context of the Koopman operator framework, both for linear and nonlinear systems. We also propose a rigorous definition of the isostables and discuss their main properties. The relation between the isostables and the Laplace averages is developed in Section \ref{sec_Laplace}. This provides an efficient algorithm for the computation of the isostables which is illustrated in Section \ref{sec_appli} for the excitable FitzHugh-Nagumo model and the Lorenz model. Finally, the related concepts of action-angle representation, global linearization, and Lyapunov function are discussed in Section \ref{sec_Lyap_lin}. Section \ref{sec_conclu} gives some concluding remarks.

\section{Isostables and Koopman operator}
\label{sec_iso}

The isostables of an asymptotically stable fixed point $\ve{x}^*$ are the sets of points that share the \emph{same asymptotic convergence} toward the fixed point. More precisely, trajectories with an initial condition on an isostable $\mathcal{I}_{\tau_0}$ simultaneously intersect the successive isostables $\mathcal{I}_{\tau_n}$ after a time interval $\tau_n-\tau_0$, thereby approaching the fixed point synchronously (Figure \ref{isostable}). The isostables partition the basin of attraction of the fixed point and define a new coordinate $\tau$ that satisfies $\dot{\tau}=1$ along the trajectories. Or equivalently, they define a coordinate $r \triangleq \exp(\lambda \tau)$ with the linear dynamics $\dot{r}=\lambda r$. This new coordinate can be used in a context of model reduction.

At this point, it is important to remark that this (intuitive) definition of isostable is not complete. Indeed, there exist an infinity of families of sets that satisfy the above-described property. But among these families, only one defines a smooth change of coordinates and is relevant to capture the asymptotic behavior of the trajectories. In this section, we will give a rigorous definition of this unique family of isostables. To do so, we first consider the particular case of linear systems. Then, we extend the concept to nonlinear systems, using the Koopman operator framework.

\begin{figure}[h]
\begin{center}
\includegraphics[width=5cm]{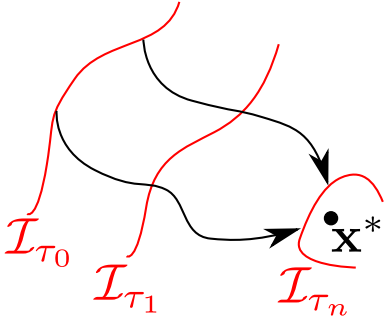}
\caption{Trajectories starting from the same isostable $\mathcal{I}_{\tau_0}$ are characterized by the same convergence toward the fixed point. They simultaneously intersect the successive isostables $\mathcal{I}_{\tau_n}$ and approach the fixed point synchronously.}
\label{isostable}
\end{center}
\end{figure}

\subsection{Linear systems}
\label{iso_linear}

Consider the stable linear system
\begin{equation}
\label{syst_lin}
\dot{\ve{x}}=\ve{A} \ve{x} \,, \quad \ve{x}\in \mathbb{R}^n\,,
\end{equation}
and assume that each eigenvalue $\lambda_j=\sigma_j+i\omega_j$ of the matrix $\ve{A}$ is of multiplicity $1$, has a strictly negative real part $\sigma_j<0$, and corresponds to the right eigenvector $\ve{v}_j$ (which is normalized, that is, $\|\ve{v}_j\|=1$). By convention, we sort the eigenvalues so that $\lambda_1$ is the eigenvalue related to the \guillemets{slowest} direction, that is
\begin{equation}
\label{hypo_eigen}
\sigma_j \leq \sigma_1 <0\,, \quad j=2,\dots,n \,.
\end{equation}
The flow induced by \eqref{syst_lin} is the continuous-time map $\phi:\mathbb{R}\times \mathbb{R}^n \mapsto \mathbb{R}^n$, that is, $\phi(t,\ve{x})$ is the solution of \eqref{syst_lin} with the initial condition $\ve{x}$. For linear systems, the flow is given by
\begin{equation}
\label{sol_lin}
\phi(t,\ve{x})=\sum_{j=1}^n s_j(\ve{x}) \ve{v}_j \, e^{\lambda_j t}\,,
\end{equation}
where $s_j(\ve{x})$ are the coordinates of the vector $\ve{x}$ in the basis $(\ve{v}_1,\dots,\ve{v}_n)$. The function $s_j(\ve{x})$ can be computed as the inner product $s_j(\ve{x})=\langle \ve{x},\ve{\tilde{v}}_j \rangle$, with $\ve{\tilde{v}}_j$ the eigenvectors of the adjoint $\ve{A}^*$, associated with the eigenvalues $\lambda^c_j=\sigma_j-i\omega_j$ and normalized so that $\langle \ve{v}_j,\ve{\tilde{v}}_j\rangle=1$. (Note that $s_j(\ve{x})$ is an eigenfunction of the so-called Koopman operator; see Section \ref{iso_nonlinear}.)

Next, we show that the isostables of linear systems are simply defined as the level sets of $|s_1(\ve{x})|=|\langle \ve{x},\ve{\tilde{v}}_1 \rangle|$. We consider separately the cases $\lambda_1$ real (with other eigenvalues real or complex)  and $\lambda_1$ complex (with other eigenvalues real or complex).

\subsubsection{Real eigenvalue $\lambda_1$}
\label{subsub_real}

When the eigenvalue $\lambda_1=\sigma_1$ is real, the trajectories induced by the flow \eqref{sol_lin} asymptotically approach the fixed point along the slowest direction $\ve{v}_1$ (since the eigenvalues are sorted according to \eqref{hypo_eigen}). Then, the trajectories characterized by the \emph{same coefficient} $|s_1(\ve{x})|\triangleq\exp(\sigma_1 \tau(\ve{x}))$ exhibit the same asymptotic convergence toward the fixed point:
\begin{equation}
\label{asympt_real}
\phi(t,\ve{x})=\ve{v}_1^{\pm} e^{\sigma_1(t+\tau(\ve{x}))}+\sum_{j=2}^n s_j(\ve{x}) \, \ve{v}_j \, \exp(\lambda_j t) \approx \ve{v}_1^{\pm} e^{\sigma_1(t+\tau(\ve{x}))} \quad \textrm{as } t\rightarrow \infty\,,
\end{equation}
where the notation $\ve{v}_1^{\pm}$ implies that either the vector $\ve{v}_1$ or $-\ve{v}_1$ must be considered. The initial conditions $\ve{x}$ of these trajectories therefore belong to the same isostable
\begin{equation}
\label{isostable_real}
\mathcal{I}_\tau=\left\{\ve{x}\in\mathbb{R}^n\Big|\ve{x}=e^{\sigma_1 \tau} \ve{v}_1^{\pm}  + \sum_{j=2}^n \alpha_j \, \ve{v}_j\,, \, \forall \alpha_j\in\mathbb{R}\right\}\,,
\end{equation}
which is obtained by considering $t=0$ in \eqref{asympt_real}. In this case, the isostables are the $(n-1)$-dimensional hyperplanes parallel to $\ve{v}_j$ for all $j>2$ (or equivalently, perpendicular to $\tilde{\ve{v}}_1$) (Figure \ref{linear_real}).

\begin{figure}[h]
\begin{center}
\subfigure[]{\includegraphics[width=5cm]{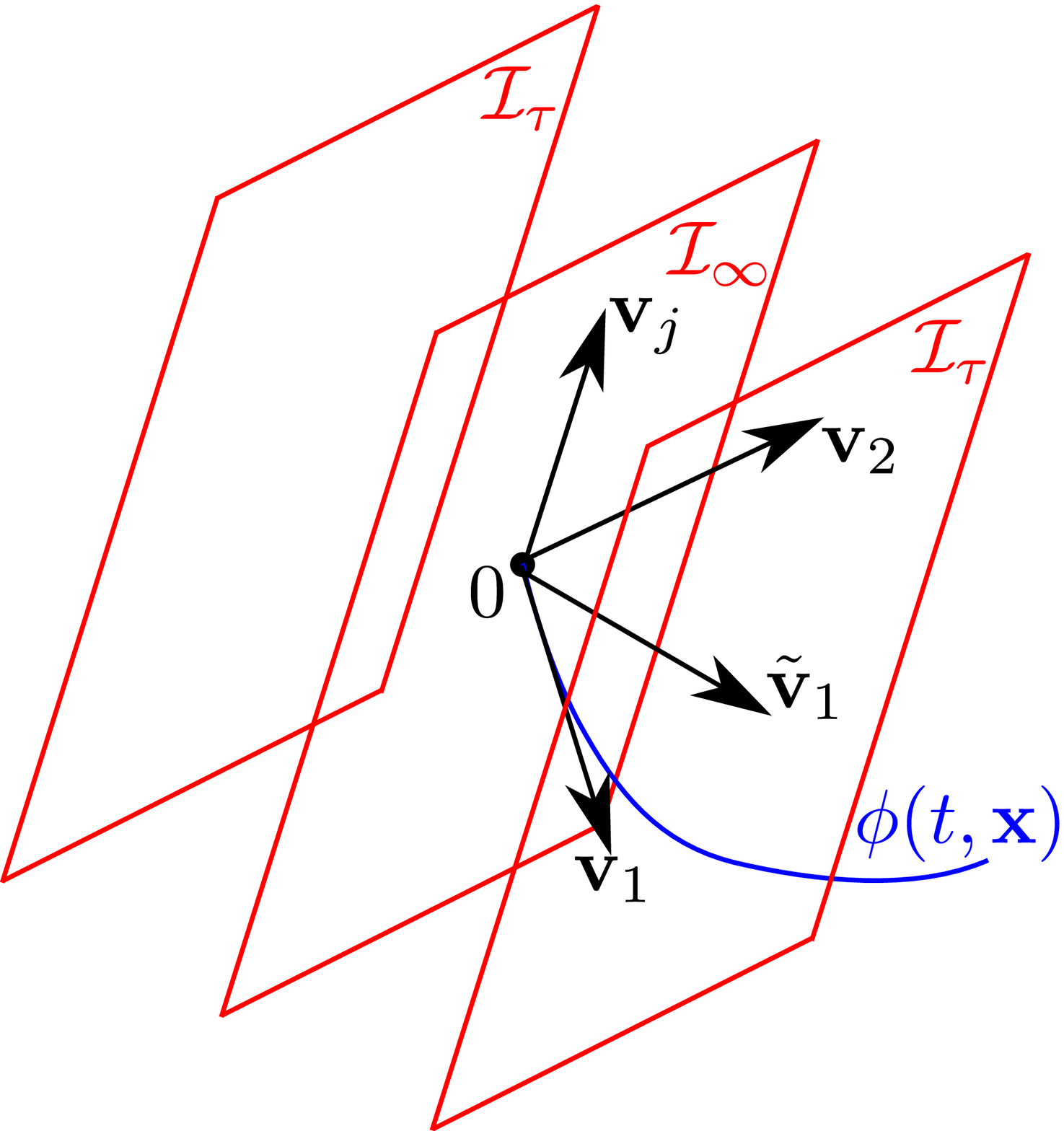}}
\hspace{1cm}
\subfigure[]{\includegraphics[width=5cm]{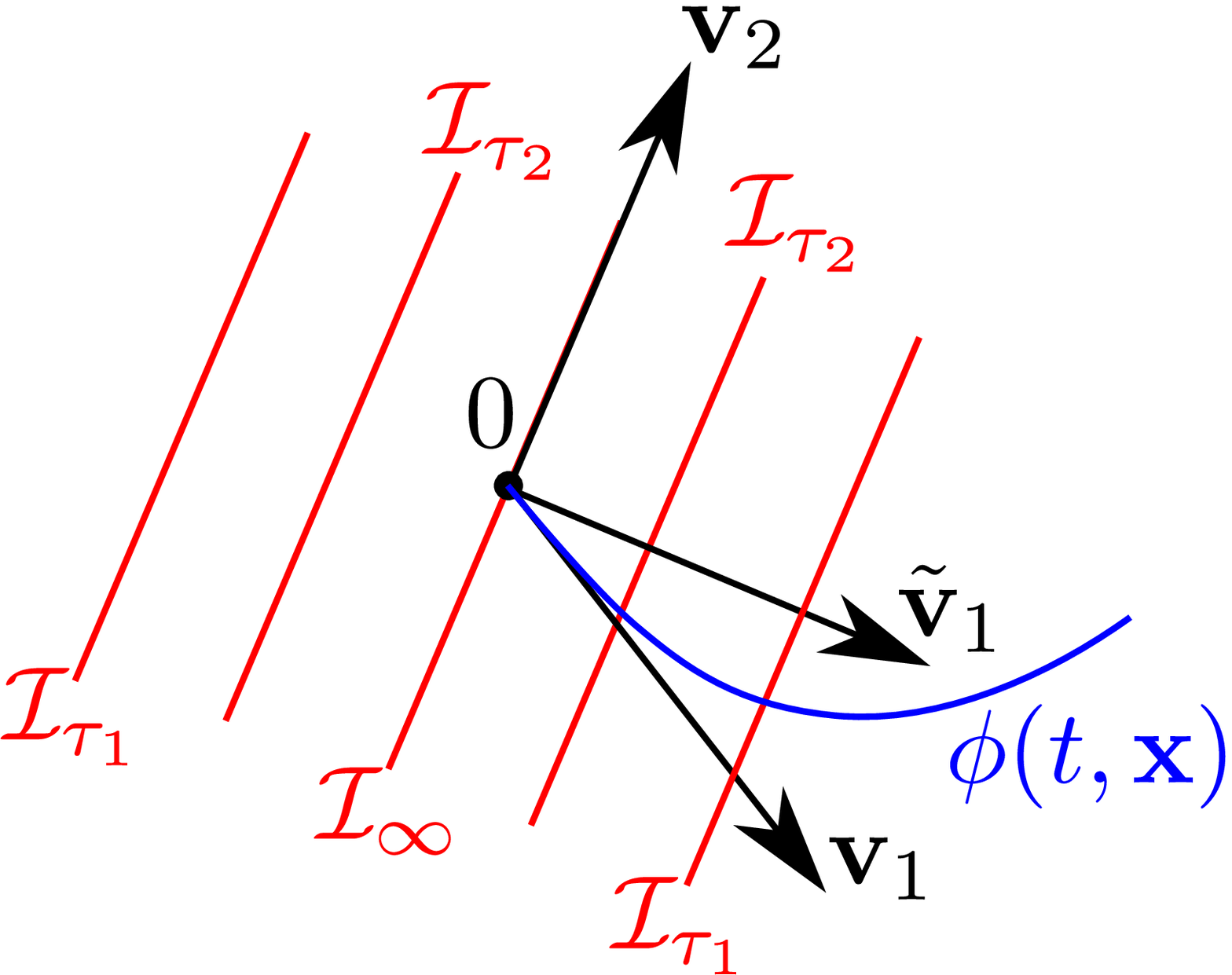}}
\caption{(a) The isostables of linear systems with a real eigenvalue $\lambda_1$ are the hyperplanes spanned by the eigenvectors $\ve{v}_j$, with $j>2$. The particular isostable $\mathcal{I}_\infty$ contains the fixed point. (b) For two-dimensional systems (or in the plane $\ve{v}_1-\ve{v}_2$), the isostables are pairs of parallel lines.}
\label{linear_real}
\end{center}
\end{figure}

\subsubsection{Complex eigenvalue $\lambda_1$}
\label{subsub_complex}

A system having a complex eigenvalue $\lambda_1$ can be transformed through the use of action-angle coordinates. Then, the isostables are obtained from the isostables \eqref{isostable_real} of the subsystem which is related to the action coordinates and which is only characterized by real eigenvalues $\sigma_j$. Consider a linear coordinate transformation that expresses the dynamics \eqref{syst_lin} in the (spectral) basis given by the vectors $\ve{v}_j$ (for $\lambda_j$ real) and $\Re\{\ve{v}_j\}$, $-\Im\{\ve{v}_j\}$ (for $\lambda_j=\lambda_{j+1}^c$ complex). (Note that $\Re\{\ve{v}_j\}$ and $\Im\{\ve{v}_j\}$ are not parallel since the two eigenvectors $\ve{v}_j$ and $\ve{v}_{j+1}$ are independent.) This is performed by diagonalizing $\ve{A}$ and by using the linear transformation
\begin{equation*}
\ve{T}=\left( \begin{array}{cc} 1 & 1\\-i & i \end{array} \right)
\end{equation*}
in each subspace spanned by a pair of complex eigenvectors ($\ve{v}_j$,$\ve{v}_{j+1}$). The dynamics become
\begin{equation*}
\begin{cases}
\dot{y}_j=\sigma_j y_j & j \in \{i \in \{1,\dots,n\} | \lambda_i \in \mathbb{R}\}\,, \\
\left(\begin{array}{c}\dot{y}_j \\ \dot{y}_{j+1}  \end{array} \right)
= \left(\begin{array}{cc} \sigma_j & -\omega_j \\ \omega_j & \sigma_j \end{array}
\right) \left(\begin{array}{c} y_j \\ y_{j+1}  \end{array} \right) & j \in \{i \in \{1,\dots,n\} | \lambda_i=\lambda_{i+1}^c \notin \mathbb{R}\} \,,
\end{cases}
\end{equation*}
with the initial conditions $y_j(0)=s_j(\ve{x}_0)$ (for $\lambda_j\in\mathbb{R}$) and $(y_j(0),y_{j+1}(0))=(2\Re\{s_j(\ve{x}_0)\},2\Im\{s_j(\ve{x}_0)\})$ (for $\lambda_j=\lambda_{j+1}^c\notin\mathbb{R}$). Then, using the variables $r_j=y_j$ (for $\lambda_j\in\mathbb{R}$) and the polar coordinates $(y_j,y_{j+1})=(r_j \cos(\theta_j), r_j \sin(\theta_j))$ (for $\lambda_j=\lambda_{j+1}^c\notin\mathbb{R}$), we obtain the canonical equations
\begin{eqnarray}
\dot{r}_j & = & \sigma_j r_j \qquad j \in \{i \in \{1,\dots,n\} | \lambda_i \in \mathbb{R} \textrm{ or } \lambda_i=\lambda_{i+1}^c \notin \mathbb{R} \} \,, \label{dyn_r}\\
\dot{\theta}_j & = & \omega_j \qquad j \in \{ i \in \{1,\dots,n\} | \lambda_i=\lambda_{i+1}^c \notin \mathbb{R} \} \,. \label{dyn_theta}
\end{eqnarray}
The initial conditions are given by $r_j(0)=s_j(\ve{x}_0)$ (for $\lambda_j\in\mathbb{R}$) and $(r_j(0),\theta_j(0))=(2|s_j(\ve{x}_0)|,\angle s_j(\ve{x}_0))$ (for $\lambda_j=\lambda_{j+1}^c\notin\mathbb{R}$), where $\angle$ denotes the argument of a complex number.

According to \eqref{dyn_r}-\eqref{dyn_theta}, the variables $r_j$ and $\theta_j$ can be interpreted as the action-angle coordinates of the system (see \cite{Arnold}) and the convergence toward the fixed point is captured by the (action) variables $r_j$. Therefore, the isostables of \eqref{syst_lin} correspond to the isostables of the linear system \eqref{dyn_r} with the real eigenvalues $\sigma_j$. Since the highest eigenvalue is $\sigma_1$, the results of Section \ref{subsub_real} imply that the isostables are characterized by a constant value $|r_1|$, that is, they are the level sets of $|s_1(\ve{x})|$. Denoting $r_1=2|s_1(\ve{x})|\triangleq \exp(\sigma_1 \tau(\ve{x}))$ and using an expression similar to \eqref{isostable_real}, we obtain (in the variables $y_i$)
\begin{equation*}
\mathcal{I}_\tau=\left\{\ve{y}\in\mathbb{R}^n\Big|\ve{y}= (\cos(\theta) \ve{e}_1 + \sin(\theta) \ve{e}_2) e^{\sigma_1 \tau} + \sum_{j=3}^n \alpha_j \, \ve{e}_j\,, \, \forall \alpha_j\in\mathbb{R}, \, \forall\theta\in[0,2\pi)\right\}\,,
\end{equation*}
where $\ve{e}_j$ are the unit vectors of $\mathbb{R}^n$, or equivalently (in the variables $x_i$)
\begin{equation}
\label{isostable_complex}
\mathcal{I}_\tau=\left\{\ve{x}\in\mathbb{R}^n\Big|\ve{x}= (\cos(\theta) \ve{a} + \sin(\theta) \ve{b}) e^{\sigma_1 \tau} + \sum_{j=3}^n \alpha_j \, \ve{v}_j\,, \, \forall \alpha_j \in\mathbb{R}, \, \forall\theta\in[0,2\pi)\right\}\,,
\end{equation}
with $\ve{a}=\Re\{\ve{v}_1\}$ and $\ve{b}=-\Im\{\ve{v}_1\}$. In this case, the isostables are the $(n-1)$-dimensional cylindrical hypersurfaces parallel to $\ve{v}_j$ for all $j\geq 3$. The intersection of an isostable with the $2$-dimensional plane spanned by ($\ve{a},\ve{b}$) (i.e., the base of the cylinder) is an ellipse (Figure \ref{linear_complex}). Indeed, a linear transformation turns the circle in the variables $y_j$ into an ellipse in the variables $x_j$.

The trajectories starting from the same isostable converge to the fixed point along a spiral characterized by the vectors $(\ve{a},\ve{b})$, according to
\begin{equation*}
\phi(t,\ve{x}) \approx  \left(\ve{a} \cos(\omega_1 t +\theta(\ve{x}))+ \ve{b} \sin(\omega_1 t +\theta(\ve{x})) \right) e^{\sigma_1(t+\tau(\ve{x}))} \quad \textrm{as } t\rightarrow \infty\,,
\end{equation*}
with $\exp(\sigma_1 \tau(\ve{x}))=2|s_1(\ve{x})|$ and $\theta(\ve{x})=\angle s_1(\ve{x})$. Note that the phase---or angle coordinate--- $\theta$ is related to the notion of \emph{isochron} (see e.g. \cite{Gine,Sabatini} and Section \ref{sec_Lyap_lin}).

\begin{figure}[h]
\begin{center}
\subfigure[]{\includegraphics[height=5cm]{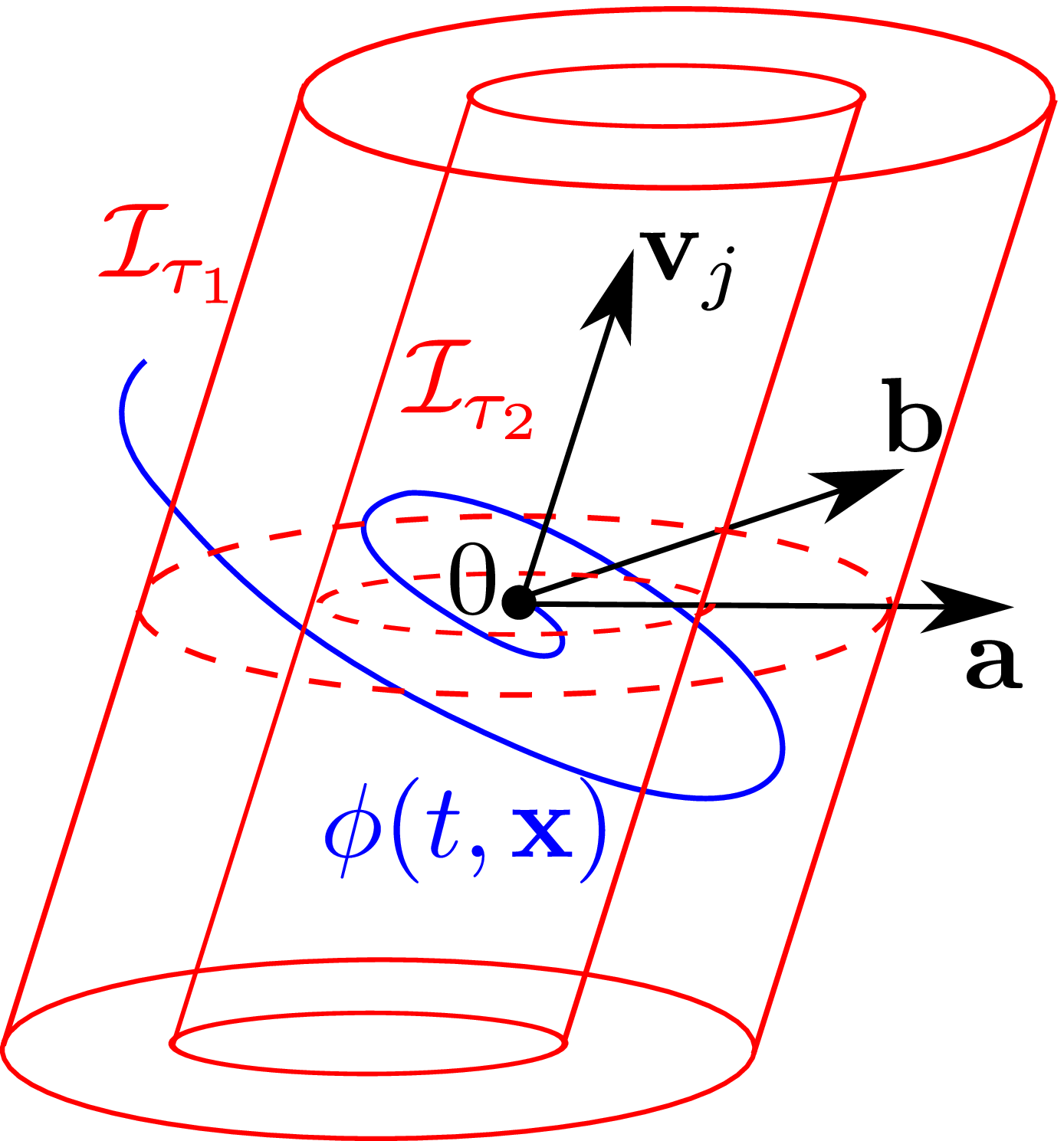}}
\hspace{1cm}
\subfigure[]{\includegraphics[height=5cm]{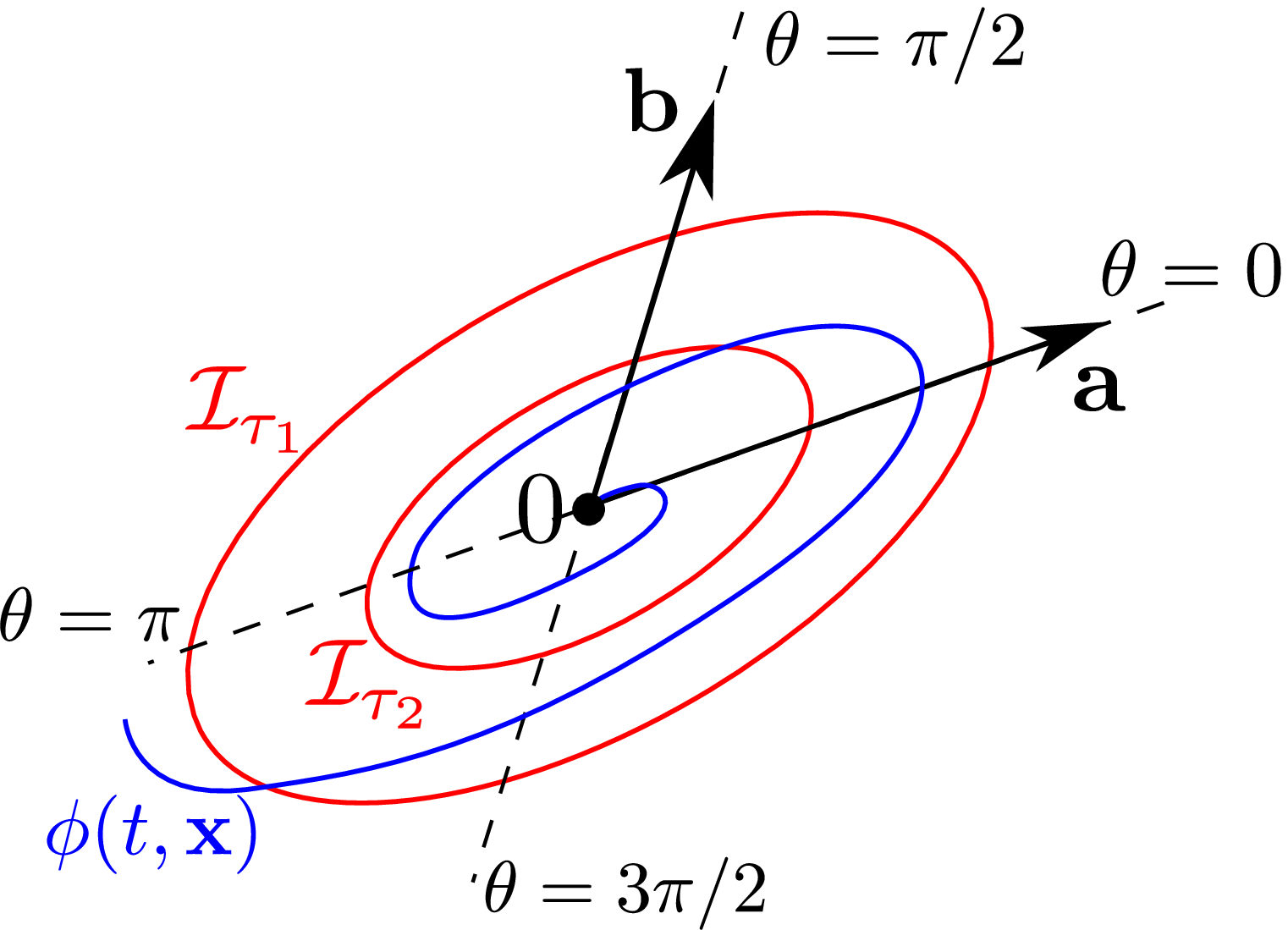}}
\caption{(a) The isostables of linear systems with a complex eigenvalue $\lambda_1$ are cylindrical hypersurfaces spanned by $\ve{v}_j$ for all $j\geq 3$. (b) For two-dimensional linear systems (or in the plane $\ve{a}-\ve{b}$), the isostables are ellipses with constant axes.}
\label{linear_complex}
\end{center}
\end{figure}

The expressions \eqref{isostable_real} and \eqref{isostable_complex} provide a unique definition of the isostables in the case of linear systems, when $\lambda_1$ is real and when $\lambda_1$ is complex, respectively. Since $\ve{v}_1^\pm=\ve{v}_1 \exp(i\theta)$ with $\theta=\{0,\pi\}$ and $\cos(\theta) \ve{a} + \sin(\theta) \ve{b}=\Re\{\ve{v}_1 \exp(i\theta)\}$, these two definitions can be summarized in a single definition.
\begin{definition}[Isostables of linear systems]
\label{def_iso_linear}
For the system \eqref{syst_lin}, the isostable $\mathcal{I}_{\tau}$ associated with the time $\tau$ is the $(n-1)$-dimensional manifold
\begin{equation*}
\mathcal{I}_{\tau} = \left \{\ve{x}\in\mathcal{B}(\ve{x}^*)\Big| \ve{x}=\Re \left\{\ve{v}_1 \, e^{i \theta} \right\} e^{\sigma_1 \tau} + \sum_{j=\overline{j}}^n \alpha_j \, \ve{v}_j\,, \, \forall \alpha_j \in\mathbb{R}, \, \forall\theta\in \Theta \right\}\,,
\end{equation*}
with $\Theta=\{0,\pi\}$ and $\overline{j}=2$ if $\lambda_1\in \mathbb{R}$, and $\Theta=[0,2\pi)$ and $\overline{j}=3$ if $\lambda_1\notin\mathbb{R}$.
\end{definition}

\subsection{Nonlinear systems}
\label{iso_nonlinear}

Now, we consider a nonlinear system
\begin{equation}
\label{syst_nonlin}
\dot{\ve{x}}=\ve{F}(\ve{x})\,, \quad \ve{x} \in \mathbb{R}^n
\end{equation}
where $\ve{F}$ is an analytic vector field, which admits a stable fixed point $\ve{x}^*$ with a basin of attraction $\mathcal{B}(\ve{x}^*)\subseteq \mathbb{R}^n$. In addition, we assume that the Jacobian matrix $\ve{J}$ computed at $\ve{x}^*$ has $n$ distinct (nonresonant) eigenvalues $\lambda_j=\sigma_j+i\omega_j$ characterized by strictly negative real parts $\sigma_j<0$ and sorted according to \eqref{hypo_eigen}. (For unstable fixed points or for multiple eigenvalues, see Remark \ref{unstable} and Remark \ref{multiple_eigenval}, respectively.) 

The isostables of linear systems have been defined as the level sets of the coefficient $s_1(\ve{x})$ that appears in the expression of the flow \eqref{sol_lin}. For nonlinear systems, an expression of the flow similar to \eqref{sol_lin} can be obtained through the framework of Koopman operator \cite{Mezic,MezicBana}. The Koopman semigroup of operators $U^t$ describes the evolution of a (vector-valued) \emph{observable} $\ve{f}:\mathbb{R}^n \mapsto \mathbb{C}^m$ along the trajectories of the system and is rigorously defined as the composition $U^t \ve{f}(\ve{x})=\ve{f} \circ \phi(t,\ve{x})$. Throughout the paper, we will make no assumption on the observables, except that they are analytic in the neighborhood of the fixed point. In the space of analytic observables, the operator has only a point spectrum and its spectral decomposition yields \cite{Mezic_ann_rev}
\begin{equation}
\label{evol_observable}
U^t \ve{f}(\ve{x}) = \sum_{\{k_1,\dots,k_n\}\in\mathbb{N}^n} s_1^{k_1}(\ve{x}) \cdots s_n^{k_n}(\ve{x}) \, \ve{\overline{v}}_{k_1\cdots k_n} \, e^{(k_1 \lambda_1+\cdots+k_n \lambda_n) t}\,.
\end{equation}
A detailed derivation of the decomposition in the case of a stable fixed point is given in Appendix \ref{appendix}. The functions $s_j(\ve{x})$, $j=1,\dots,n$, are the smooth eigenfunctions of $U^t$ associated with the eigenvalues $\lambda_j$, i.e.
\begin{equation}
\label{prop_eigenfunction}
U^t s_j(\ve{x})=s_j(\phi(t,\ve{x}))=s_j(\ve{x}) e^{\lambda_j t}\,,
\end{equation}
and the vectors $\ve{\overline{v}}_{k_1\cdots k_n}$ are the so-called Koopman modes \cite{Rowley}, i.e. the projections of the observable $\ve{f}$ onto $s_1^{k_1}(\ve{x}) \cdots s_n^{k_n}(\ve{x})$. For the particular observable $\ve{f}(\ve{x})=\ve{x}$, \eqref{evol_observable} corresponds to the expression of the flow and can be rewritten as
\begin{equation}
\label{sol_nonlin}
\phi(t,\ve{x})= U^t \ve{x} =\ve{x}^* + \sum_{j=1}^n s_j(\ve{x}) \ve{v}_j \, e^{\lambda_j t} + \sum_{\substack{\{k_1,\dots,k_n\}\in\mathbb{N}_0^n\\ k_1+\cdots+k_n>1}} s_1^{k_1}(\ve{x}) \cdots s_n^{k_n}(\ve{x}) \,\ve{v}_{k_1\cdots k_n} \, e^{(k_1 \lambda_1+\cdots+k_n \lambda_n) t}\,.
\end{equation}
The first part of the expansion is similar to the linear flow \eqref{sol_lin}. The eigenvalues $\lambda_j$ and the Koopman modes $\ve{v}_j$ are the eigenvalues and eigenvectors of $\ve{J}$, respectively. Although the eigenfunctions $s_j(\ve{x})$ are not computed as the inner products $\langle \ve{x},\tilde{\ve{v}}_j \rangle$ as in the linear case, they can be interpreted as the inner products $\langle \ve{z},\tilde{\ve{v}}_j \rangle$, where $\ve{z}$ is the initial condition of a virtual trajectory evolving according to the linearized dynamics $\dot{\ve{z}}=\ve{J}\ve{z}$ and characterized by the same asymptotic evolution as $\phi(t,\ve{x})$ \cite{Lan}. The other terms in \eqref{sol_nonlin} do not appear in the expression of the linear flow \eqref{sol_lin} and account for the transient behavior of the trajectories owing to the nonlinearity of the dynamics.

The isostables can be rigorously defined as the \emph{level sets of the absolute value of the eigenfunction} $|s_1(\ve{x})|$. Indeed, the asymptotic evolution of the flow \eqref{sol_nonlin} is dominated by the first mode associated to $\lambda_1$. Then, a same argument as in Section \ref{iso_linear} shows that the points $\ve{x}$ characterized by the same value $|s_1(\ve{x})|$ are the initial conditions of trajectories that converge synchronously to the fixed point, with the evolution
\begin{equation}
\label{asympt_behav}
\phi(t,\ve{x})\approx 
\begin{cases} \ve{x}^*+\ve{v}_1^\pm e^{\sigma_1(t+\tau(\ve{x}))}\,, \quad e^{\sigma_1 \tau(\ve{x})}=|s_1(\ve{x})|\,, & \lambda_1\in \mathbb{R}\,, \\
\ve{x}^*+ \Re \left\{\ve{v}_1 \, e^{i(\omega_1 t+\theta(\ve{x}))} \right\} e^{\sigma_1(t+\tau(\ve{x}))}\,, \quad e^{\sigma_1 \tau(\ve{x})}=2|s_1(\ve{x})|\,,\, \theta(\ve{x})=\angle s_1(\ve{x}) & \lambda_1\notin \mathbb{R}\,.
\end{cases}
\end{equation}

We are now in position to propose a general definition for the isostables of a fixed point, which is valid both for linear and nonlinear systems and which is reminiscent of the usual definition of \emph{isochrons} for limit cycles \cite{Guckenheimer_iso,Winfree2}.
\begin{definition}[Isostables]
\label{def_iso}
For the system \eqref{syst_nonlin}, the isostable $\mathcal{I}_{\tau}$ of the fixed point $\ve{x}^*$, associated with the time $\tau$, is the $(n-1)$-dimensional manifold
\begin{equation*}
\mathcal{I}_{\tau} = \left \{\ve{x}\in\mathcal{B}(\ve{x}^*)\Big|\exists \, \theta \in \Theta \textrm{ s.t. } \lim_{t\rightarrow\infty} e^{-\sigma_1 t} \left \|\phi(t,\ve{x})-\ve{x}^*- \Re \left\{\ve{v}_1 \, e^{i(\omega_1 t+\theta)} \right\} e^{\sigma_1(t+\tau)} \right \|=0 \right\}\,,
\end{equation*}
with $\Theta=\{0,\pi\}$ and $\omega_1=0$ if $\lambda_1\in\mathbb{R}$ and $\Theta=[0,2\pi)$ if $\lambda_1\notin \mathbb{R}$.
\end{definition}

The reader will easily verify that, for all $\ve{x}$ belonging to the same isostable, Definition \ref{def_iso} imposes the same value $|s_1(\ve{x})|$ in the decomposition of the flow \eqref{sol_nonlin} and the same asymptotic behavior \eqref{asympt_behav}. Note that without the multiplication by the increasing exponential $e^{-\sigma_1 t}$, one would have $\mathcal{I}_\tau=\mathcal{B}(\ve{x}^*)$ $\forall \tau$ since $\phi(t,\ve{x})-\ve{x}^*\rightarrow \ve{0}$ as $t\rightarrow \infty$ for all $\ve{x}\in \mathcal{B}(\ve{x}^*)$.

Except for the case of multiple eigenvalues, for which $\ve{v}_1$ might not be unique (see Remark \ref{multiple_eigenval}), the isostables are uniquely defined through Definition \ref{def_iso}. Uniqueness of the isostables also follows from the fact that the Koopman operator has a unique eigenfunction $s_1(\ve{x})$ which is continuously differentiable in the neighborhood of the fixed point. Since it is precisely this eigenfunction $s_1(\ve{x})$ that appears in \eqref{sol_nonlin}, the isostables are the only sets that are relevant to capture the asymptotic behavior of the trajectories.

\begin{remark}[Unstable fixed point]
\label{unstable}
Definition \ref{def_iso} is easily extended to unstable fixed points characterized by $\sigma_j>\sigma_1>0$ for all $j$. Indeed, the isostables are still given by Definition \ref{def_iso}, where the limit $t\rightarrow \infty$ is replaced by $t\rightarrow -\infty$, that is, one considers the flow $\phi(-t,\ve{x})$ induced by the (stable) backward-time system. In this case, the isostables are related to the unstable eigenfunction $s_1(\ve{x})$ of the Koopman operator.
\end{remark}
\begin{remark}[Multiple eigenvalues]
\label{multiple_eigenval}
When the eigenvalue $\lambda_1$ has a multiplicity $m>1$, the fixed point is either a star node ($m$ linearly independent eigenvectors) or a degenerate node ($m$ linearly dependent eigenvectors). In the case of a star node, Definition \ref{def_iso} is not unique since it depends on the direction of the eigenvector $\ve{v}_1$ (in other words, a $C^1$ eigenfunction of the Koopman operator corresponding to the eigenvalue $\lambda_1$ is not unique). Actually, $\ve{v}_1$ should be replaced in Definition \ref{def_iso} by any linear combination of $m$ orthonormal eigenvectors of $\lambda_1$, a situation where the isostables lying in the vicinity of the fixed point correspond to cylindrical hypersurfaces whose intersection with the hyperplane spanned by the eigenvectors of $\lambda_1$ is a hypersphere. In the case of a degenerate node, the asymptotic evolution toward the fixed point is dominated by the (slowest) term $s_1(\ve{x}) \ve{v}_1\, t^{m-1} \exp(\sigma_1 t)$. Then, the increasing exponential $\exp(-\sigma_1 t)$ in Definition \ref{def_iso} must be replaced by $t^{1-m} \exp(-\sigma_1 t)$.
\end{remark}

\subsection{Some remarks on the isostables}

\paragraph*{Equivalent definitions for excitable systems.}
In \cite{Rabinovitch}, the authors considered two-dimensional excitable systems characterized by a transient attractor (i.e. slow manifold) which attracts all the trajectories as they approach the fixed point. They defined the isostables (they actually used the term \guillemets{isochrons}, see Section \ref{subsec_isochrons}) as the sets of points that converge to the same trajectory on the transient attractor. This definition is equivalent to Definition \ref{def_iso} since both impose that trajectories on the same isostable have the same asymptotic behavior (see also Section \ref{sec_FN}). However, the definition of \cite{Rabinovitch} is qualitative since no trajectory effectively reaches the transient attractor (which may even lose its normal stability property near a fixed point with complex eigenvalues). Also, it is valid only if the system admits a transient attractor induced by the slow-fast dynamics. In contrast, Definition \ref{def_iso} is more general and does not rely on the existence of a transient attractor.

For systems with a slow (or center) manifold, the \guillemets{projection manifolds} studied in \cite{Roberts1,Roberts2} are related to the isostables. They are the sets of initial conditions for which the trajectories share the same long-term behavior on the slow manifold. In addition, they can be obtained through the normal form of the dynamics \cite{Roberts2}. If the slow manifold is one-dimensional and if $\lambda_1$ is real, the projection manifolds are identical to the isostables. Otherwise, they do not exactly correspond to the isostables since they are not related to the slowest direction $\ve{v}_1$ only and are not of codimension-$1$.

\paragraph*{Isostables and flow.}
The flow $\phi(\Delta t,\cdot)$ maps the isostable $\mathcal{I}_\tau$ to the isostable $\mathcal{I}_{\tau+\Delta t}$, for all $\Delta t \in \mathbb{R}$ (as explained in the beginning of Section \ref{sec_iso}). Indeed, if $\ve{x}\in\mathcal{I}_\tau$, Definition \ref{def_iso} implies that
\begin{equation*}
\lim_{t\rightarrow\infty} e^{-\sigma_1 t} \left \|\phi(t,\ve{x})-\ve{x}^*- \Re \left\{\ve{v}_1 \, e^{i(\omega_1 t+\theta)} \right\} e^{\sigma_1(t+\tau)} \right \|=0
\end{equation*}
for some $\theta\in\Theta$. Using the substitution $t=t'+\Delta t$, we have
\begin{equation*}
\lim_{t'\rightarrow\infty} e^{-\sigma_1 t'} \left \|\phi\left( t',\phi(\Delta t,\ve{x})\right) -\ve{x}^*- \Re \left\{\ve{v}_1 \, e^{i(\omega_1 t'+\theta')} \right\} e^{\sigma_1(t'+\tau+\Delta t)} \right \|=0\,,
\end{equation*}
with $\theta'=\theta+\omega_1 \Delta t\in\Theta$, so that $\phi(\Delta t,\ve{x})\in \mathcal{I}_{\tau+\Delta t}$.

\paragraph*{Local geometry near the fixed point.}
The isostables close to the fixed point have a geometry similar to the isostables of the linearized dynamics, i.e. parallel hyperplanes ($\lambda_1\in\mathbb{R}$) or cylindrical hypersurfaces with constant axes of the elliptical sections ($\lambda_1\notin\mathbb{R}$) (see Section \ref{iso_linear}). This follows from the fact that, in the vicinity of the fixed point, the flow \eqref{sol_nonlin} and the flow induced by the linearized dynamics are (approximately) equal, so that their eigenfunctions $s_1(\ve{x})$ have (approximately) the same value for $\|\ve{x}-\ve{x}^*\| \ll 1$ (see also \eqref{approx_eigen} in Appendix \ref{appendix}).

\paragraph*{Invariant fibration.}

When the eigenvalue $\lambda_1$ is real, the isostables are the invariant fibers of the $1$-dimensional invariant manifold $V$ defined as the trajectory associated with the slow direction $\ve{v}_1$ (i.e. the transient attractor in the case of slow-fast systems). Given their local geometry, it is clear that the isostables near the fixed point are the fibers defined by the splitting $N \oplus TV$, where $N=\textrm{span}\{\ve{v}_2,\dots,\ve{v}_n\}$ and $TV=\textrm{span}\{\ve{v}_1\}$. Moreover, it follows from the invariance property of the isostables that this local fibration is naturally extended to the whole invariant manifold $V$ by backward integration of the flow. Provided that $\sigma_2<\sigma_1$, the normal hyperbolicity of $V$ implies that the isostables are characterized by smoothness properties and persist under a small perturbation of the vector field \cite{Fenichel1,Hirsch}. In addition, this description also implies the uniqueness of the concept of isostables. Note that Definition \ref{def_iso} is recovered in \cite{Fenichel2}, Theorem 3, and corresponds to the property that the points on the same fiber converge to a trajectory on $V$ with the fastest rate.

When $\lambda_1$ is complex, however, the isostables cannot be interpreted as the invariant fibers of an invariant manifold. They are homeomorphic to a circle (or to a cylinder) and cannot be the sets of points converging to the same trajectory, since the flow is continuous. Moreover, in the neighborhood of the fixed point, one observes no particular one-dimensional invariant manifold (e.g. a slow manifold) that is tangent to the $\Re\{\ve{v}_1\}-\Im\{\ve{v}_1\}$ plane. In that case, the only definition of the isostables is in terms of an eigenfunction of the Koopman operator. 

\paragraph*{Extension to other eigenfunctions.}

The isostables $\mathcal{I}_\tau$ are related to the first eigenfunction $s_1(\ve{x})$ of the Koopman operator, but the concept can be directly generalized to other eigenfunctions. Namely, the sets $\mathcal{I}^{(j)}_{\tau^{(j)}}$, $j\in \mathcal{J}=\{i \in \{1,\dots,n\} | \lambda_{i} \in \mathbb{R} \textrm{ or } \lambda_i=\lambda_{i+1}^c \notin \mathbb{R} \}$, are obtained by considering the level sets of $|s_j(\ve{x})|$. The extension is useful to derive an action-angle coordinates representation of the system, to perform a global linearization of the dynamics (see Section \ref{subsec_linearization}), or to compute the (un)stable manifold of an attractor.

The intersection between the sets $\mathcal{I}^{(j)}_{\tau^{(j)}}$, with $j\leq \overline{j}$, is defined as the generalization of Definition \ref{def_iso}
\begin{equation}
\label{intersection}
\begin{split}
\bigcap_{\substack{j\in\mathcal{J} \\ j\leq \overline{j}}} \mathcal{I}^{(j)}_{\tau^{(j)}} =\Bigg \{ & \ve{x}\in\mathcal{B}(\ve{x}^*)\Big|\exists \, \theta_j \in \Theta_j \textrm{ s.t. } \\
& \lim_{t\rightarrow\infty} e^{-\sigma_{\overline{j}} t} \Bigg \|\phi(t,\ve{x})-\ve{x}^* - \sum_{\substack{j\in\mathcal{J} \\ j\leq \overline{j}}} \Re \left\{ \ve{v}_j \, e^{i(\omega_j t+\theta_j)} \right\} e^{\sigma_j(t+\tau^{(j)})} \Bigg \|=0 \Bigg \} \,,
\end{split}
\end{equation}
with $\Theta_j=\{0,\pi\}$ if $\lambda_j\in\mathbb{R}$ and $\Theta_j=[0,2\pi)$ if $\lambda_j\notin \mathbb{R}$. When $\tau^{(j)}=\infty$ for all $j<\overline{j} \in \mathcal{J}$, \eqref{intersection} is equivalent to Definition \ref{def_iso}, so that it can be interpreted as an isostable for the system restricted to the invariant manifold $M_{\overline{j}}=\bigcap_{j\in\mathcal{J},j< \overline{j}} \mathcal{I}^{(j)}_{\tau^{(j)}=\infty}$. (The manifold $M_{\overline{j}}$ is associated with the fast directions $\ve{v}_{j}$, $j=\overline{j},\dots,n$.) In addition, if $\lambda_{\overline{j}}\in \mathbb{R}$, \eqref{intersection} defines a codimension-$\overline{j}$ invariant fibration of the invariant manifold $V_{\overline{j}}=\bigcap_{j\in\mathcal{J},j > \overline{j}} \mathcal{I}^{(j)}_{\tau^{(j)}=\infty}$. (The manifold $V_{\overline{j}}$ is associated with the slow directions $\ve{v}_{j}$, $j=1,\dots,\overline{j}$.) If $V_{\overline{j}}$ is a slow manifold, then the fibration \eqref{intersection} corresponds to the projection manifolds considered in \cite{Roberts1, Roberts2}. Note that the family of manifolds $V_{\overline{j}}$ generalizes the notion of slow manifold observed for systems with slow-fast dynamics.

\section{Laplace averages}
\label{sec_Laplace}

In this section, we show that the isostables can be obtained through the computation of the so-called \emph{Laplace averages}. The Laplace averages of a scalar observable $f:\mathbb{R}^n\mapsto \mathbb{C}$ are given by
\begin{equation}
\label{Laplace_average}
f^*_{\lambda}(\ve{x})=\lim_{T\rightarrow \infty} \frac{1}{T} \int_0^T (f \circ \phi_t)(\ve{x})\, e^{-\lambda t} \, dt\,,
\end{equation}
with $\phi_t(\ve{x})=\phi(t,\ve{x})$ and $\lambda\in\mathbb{C}$. (The observable $f$ has to satisfy some conditions which ensure that the averages exist.) When it exists and is nonzero for some $\lambda$ and $f$, the Laplace average $f^*_{\lambda}(\ve{x})$ corresponds to the eigenfunction of the Koopman operator associated with the eigenvalue $\lambda$ \cite{Mezic_ann_rev}. Indeed, one easily verifies that
\begin{equation*}
\begin{split}
U^{t'} f^*_{\lambda}(\ve{x}) & = \lim_{T\rightarrow \infty} \frac{1}{T} \int_0^T (f \circ \phi_{t+t'})(\ve{x}) \, e^{-\lambda t} \, dt\\
& = e^{\lambda t'} \, \lim_{T\rightarrow \infty} \frac{1}{T} \int_{t'}^{T+t'} (f \circ \phi_{t})(\ve{x})\, e^{-\lambda t} \, dt \\
& = e^{\lambda t'} \, f^*_{\lambda}(\ve{x})
\end{split}
\end{equation*}
where the second equality is obtained by substitution. For systems with a stable fixed point, the Laplace average $f^*_{\lambda_1}(\ve{x})$ corresponds (up to a scalar factor) to the eigenfunction $s_1(\ve{x})$, and is therefore related to the concept of isostable. In addition, the Laplace averages are an extension of the Fourier averages \cite{Mezic,MezicBana} that were used in \cite{Mauroy_Mezic} to compute the isochrons of limit cycles.

\begin{remark}
\label{gen_Laplace_av}
Instead of \eqref{Laplace_average}, the generalized Laplace averages \cite{Mezic_ann_rev}
\begin{equation*}
f^*_{\lambda_j}(\ve{x})=\lim_{T\rightarrow \infty} \frac{1}{T} \int_0^T \left( (f \circ \phi_t)(\ve{x})-f(\ve{x}^*)-\sum_{k=1}^{j-1} f^*_{\lambda_k}(\ve{x}) e^{\lambda_k t} \right)\, e^{-\lambda_j t} \, dt
\end{equation*}
must be considered to obtain other eigenfunctions $s_j(\ve{x})$, $j\geq 2$, and the associated sets $\mathcal{I}^{(j)}_{\tau^{(j)}}$ considered in \eqref{intersection}. However, their computation is delicate since it requires a very accurate computation of the other (generalized) Laplace averages $f^*_{\lambda_k}(\ve{x})$, $k<j$, and goes beyond the scope of the present paper.
\end{remark}

\subsection{The main result}

The exact connection between the Laplace averages and the isostables is given in the following proposition.
\begin{proposition}
\label{prop_iso_average}
Consider an observable $f\in C^1$ such that $f(\ve{x}^*)=0$ and $\left\langle \nabla f(\ve{x}^*), \ve{v}_1\right\rangle \neq 0$. Then, a unique level set of the Laplace average $|f^*_{\lambda_1}|$ corresponds to a unique isostable. That is, $|f^*_{\lambda_1}(\ve{x})|=|f^*_{\lambda_1}(\ve{x}')|$, with $\ve{x}\in \mathcal{I}_\tau$ and $\ve{x}'\in \mathcal{I}_{\tau'}$, if and only if $\tau=\tau'$. In addition,
\begin{equation*}
\tau-\tau'=\frac{1}{\sigma_1} \ln \left|\frac{f^*_{\lambda_1}(\ve{x})}{f^*_{\lambda_1}(\ve{x}')}\right|\,.
\end{equation*}
\end{proposition}
\begin{proof}
If $\ve{x}$ belongs to the isostable $\mathcal{I}_\tau$, one has, for some $\theta \in \Theta$,
\begin{eqnarray}
&& \lim_{t \rightarrow \infty} e^{-\sigma_1 t} \left| (f\circ \phi_t)(\ve{x}) - f(\ve{x}^*)-\left\langle \nabla f(\ve{x}^*), \Re \left\{\ve{v}_1 \, e^{i(\omega_1 t+\theta)} \right\} \right\rangle e^{\sigma_1(t+\tau)} \right| \nonumber \\
&& \quad =\lim_{t \rightarrow \infty} e^{-\sigma_1 t} \left|\left\langle \nabla f(\ve{x}^*),\phi_t(\ve{x}) -\ve{x}^*\right\rangle +o(\|\phi_t(\ve{x}) -\ve{x}^*\|) -\left\langle \nabla f(\ve{x}^*), \Re \left\{\ve{v}_1 \, e^{i(\omega_1 t+\theta)} \right\} \right\rangle e^{\sigma_1(t+\tau)} \right| \nonumber \\
&& \quad \leq \|\nabla f(\ve{x}^*)\| \lim_{t \rightarrow \infty} e^{-\sigma_1 t} \left\|\phi_t(\ve{x})-\left(\ve{x}^*+\Re \left\{\ve{v}_1 \, e^{i(\omega_1 t+\theta)} \right\}\,e^{\sigma_1 (t+\tau)}\right)\right\| \nonumber \\
&& \quad \qquad +\lim_{t \rightarrow \infty} e^{-\sigma_1 t} o \Big(\Big\| \sum_{j=1}^n s_j(\ve{x}) \ve{v}_j \, e^{\lambda_j t} + \sum_{\substack{\{k_1,\cdots,k_n\}\in\mathbb{N}_0^n\\ k_1+\cdots+k_n>1}} s_1^{k_1}(\ve{x}) \cdots s_n^{k_n}(\ve{x}) \,\ve{v}_{k_1\cdots k_n} \, e^{(k_1 \lambda_1+\cdots+k_n \lambda_n) t} \Big\|\Big) \nonumber \\
&& \quad =0 \label{inequa1}
\end{eqnarray}
with $\lambda_1=\sigma_1+i \omega_1$. The first equality is obtained through a first-order Taylor approximation, the inequality results from the Cauchy–Schwarz inequality and the expression of the flow \eqref{sol_nonlin}, and the last equality is implied by Definition \ref{def_iso}. Then, it follows from \eqref{inequa1} that
\begin{align*}
\begin{split}
&\left|\lim_{T\rightarrow \infty} \frac{1}{T} \int_0^T (f \circ \phi_t)(\ve{x})\, e^{-\lambda_1 t} \,dt - \lim_{t\rightarrow \infty} \frac{1}{T} \int_0^T \left(f(\ve{x}^*)+\left\langle \nabla f(\ve{x}^*), \Re \left\{\ve{v}_1 \, e^{i(\omega_1 t+\theta)} \right\} \right\rangle e^{\sigma_1(t+\tau)} \right) \, e^{-\lambda_1 t} \,dt \right|\\
& \leq \lim_{T \rightarrow \infty} \frac{1}{T} \int_0^T e^{-\sigma_1 t}  \left| (f\circ \phi_t)(\ve{x}) - f(\ve{x}^*)- \left\langle \nabla f(\ve{x}^*), \Re \left\{\ve{v}_1 \, e^{i(\omega_1 t+\theta)} \right\} \right\rangle e^{\sigma_1(t+\tau)} \right| dt=0 \,,
\end{split}
\end{align*}
or equivalently, given \eqref{Laplace_average} and since $f(\ve{x}^*)=0$,
\begin{eqnarray}
f^*_{\lambda_1}(\ve{x}) && =\lim_{T\rightarrow \infty} \frac{1}{T} \int_0^T \left(f(\ve{x}^*)+\left\langle\nabla f(\ve{x}^*) , \Re \left\{\ve{v}_1 \, e^{i(\omega_1 t+\theta)} \right\} \right\rangle e^{\sigma_1(t+\tau)} \right) \, e^{-\lambda_1 t} \, dt \nonumber\\
&& =\lim_{T\rightarrow \infty} \frac{1}{T} \int_0^T \left\langle \nabla f(\ve{x}^*) , \frac{\ve{v}_1 \, e^{i(\omega_1 t+\theta)}+ \ve{v}^c_1 \, e^{-i(\omega_1 t+\theta)}}{2} \right\rangle e^{\sigma_1 \tau -i\omega_1 t} \, dt \nonumber \\
&& =\lim_{T\rightarrow \infty} \frac{1}{2T} \left(\int_0^T \left\langle \nabla f(\ve{x}^*) , \ve{v}_1 \right\rangle \, e^{\sigma_1 \tau+ i\theta} dt + \int_0^T \left\langle \nabla f(\ve{x}^*) ,\ve{v}^c_1 \right\rangle \, e^{\sigma_1 \tau -i(2\omega_1 t+\theta)} dt \, \right). \label{average_gen}
\end{eqnarray}
If $\lambda_1\in \mathbb{R}$, one has $\omega_1=0$, $\ve{v}_1=\ve{v}_1^c$, and $e^{i\theta}=e^{-i\theta}$ (since $\theta \in \Theta=\{0,\pi\}$). Then, it follows from \eqref{average_gen} that
\begin{equation}
\label{real_average_value}
f^*_{\lambda_1}(\ve{x})=\lim_{T\rightarrow \infty} \frac{1}{T} \int_0^T \left\langle \nabla f(\ve{x}^*) , \ve{v}_1  \right\rangle \, e^{\sigma_1 \tau+ i\theta} dt= \left\langle \nabla f(\ve{x}^*) , \ve{v}_1 \right\rangle \, e^{\sigma_1 \tau+ i\theta}
\end{equation}
and 
\begin{equation}
\label{abs_average_real}
|f^*_{\lambda_1}(\ve{x})|=|\left\langle \nabla f(\ve{x}^*) , \ve{v}_1 \right\rangle| \, e^{\sigma_1 \tau} \,, \quad \lambda_1\in\mathbb{R}\,.
\end{equation}
If $\lambda_1 \notin \mathbb{R}$, $\omega_1\neq 0$ implies that the second term of \eqref{average_gen} is equal to zero, which yields
\begin{equation}
\label{abs_average_complex}
|f^*_{\lambda_1}(\ve{x})|=\frac{|\left\langle \nabla f(\ve{x}^*) , \ve{v}_1 \right\rangle|}{2} \, e^{\sigma_1 \tau} \,, \quad \lambda_1\notin\mathbb{R} \,.
\end{equation}
For $\ve{x}'\in \mathcal{I}_{\tau'}$, the inequalities \eqref{abs_average_real} or \eqref{abs_average_complex} still hold (with $\tau$ replaced by $\tau'$), so that the result follows provided that $\left\langle \nabla f(\ve{x}^*), \ve{v}_1 \right\rangle \neq 0$.
\end{proof}

The Laplace average $f_{\lambda_1}^*(\ve{x})$ considered in Proposition \ref{prop_iso_average} actually extracts the term $\overline{v}_{10\cdots0} \, s_1(\ve{x})$ from the expression of $U^t f(\ve{x})$ \eqref{evol_observable}. The Koopman mode $\overline{v}_{10\cdots0}$ corresponds to $\left\langle \nabla f(\ve{x}^*), \ve{v}_1 \right\rangle$, as shown by \eqref{abs_average_real} and \eqref{abs_average_complex} (recall that $s_1=\exp(\sigma_1 \tau)$ when $\lambda_1 \in \mathbb{R}$ or $s_1=\exp(\sigma_1 \tau)/2$ when $\lambda_1 \notin \mathbb{R}$). This value must be nonzero to ensure that $f$ has a nonzero projection onto $s_1$.

\begin{remark}[Unstable fixed point and multiple eigenvalues (see also Remarks \ref{unstable} and \ref{multiple_eigenval})]
(i) For unstable fixed points with $\sigma_j>\sigma_1>0$ for all $j$, the isostables are the level sets of the Laplace averages $|f^*_{-\lambda_1}|$ computed for backward-in-time trajectories $\phi(-t,\cdot)$.\\
(ii) In the case of a star node (e.g. with a real eigenvalue of multiplicity $m$), the isostables obtained through the Laplace averages depend on the choice of the observable $f$, which can have a nonzero projection $\left\langle \nabla f(\ve{x}^*) , \ve{v}_j \right\rangle$, $j=1,\dots,m$, on several eigenfunctions of the Koopman operator associated with the eigenvalue $\lambda_1$. However, a unique family of isostables is obtained by considering the level sets of $\sqrt{\sum_{k=1}^m (f^*_{\lambda_1,k})^2}$, where $f^*_{\lambda_1,k}$ denotes the Laplace average for an observable $f_k$ that satisfies $\left\langle \nabla f_k(\ve{x}^*) , \ve{v}_j \right\rangle=0$ for all $j\in\{1,\dots, m\}\setminus \{k\}$.\\
(iii) In the case of a degenerate fixed point (eigenvalue of multiplicity $m$), the isostables are computed with the Laplace averages, but the exponential $\exp(-\lambda_1 t)$ in \eqref{Laplace_average} must be replaced by $t^{1-m} \exp(-\lambda_1 t)$.
\end{remark}

\subsection{Numerical computation of the Laplace averages}
\label{sec_num_comput}

Proposition \ref{prop_iso_average} shows the strong connection between the isostables and the Laplace averages, a result which provides a straightforward method for computing the isostables. Similarly to the method developed in \cite{Mauroy_Mezic}, the computation of isostables is realized in two steps: (i) the Laplace averages are computed (over a finite time horizon) for a set of sample points (distributed on a regular grid or randomly); (ii) the level sets of the Laplace averages (i.e. the isostables) are obtained using interpolation techniques. The proposed method is flexible and well-suited to the use of adaptive grids, for instance. In addition, the averages can be computed either in the whole basin of attraction of the fixed point or only in regions of interest.

It is important to note that the computation of the Laplace averages involves the multiplication of the very small quantity $(f\circ \phi_t)(\ve{x})$ with the very large quantity $\exp(-\lambda_1 t)$, as $t\rightarrow \infty$. When the trajectory approaches the fixed point, the relative error of the integration method implies that the (numerically computed) quantity $(f\circ \phi_t)(\ve{x})$ does not compensate exactly the value $\exp(-\lambda_1 t)$, and the computation becomes numerically unstable. Therefore, a high accuracy of the numerical integration scheme and a reasonably small time horizon $T$ are required for the computation of the Laplace averages.

In spite of the numerical issue mentioned above, an algorithm based on a straightforward calculation of the Laplace averages produces good results. However, it is improved if one can avoid computing the integral. Toward this end, we remark that evaluating the integral \eqref{Laplace_average} is not necessary when $\lambda_1$ is real, since the integrand converges to a constant value. When $\lambda_1$ is complex, we consider the successive iterations of the discrete time-$T_1$ map $\phi(T_1,\cdot)$, with $T_1=2\pi/\omega_1$. The result is summarized as follows.
\begin{proposition}
\label{prop_limit}
(i) Real eigenvalue $\lambda_1$. Consider an observable $f\in C^1$ that satisfies $f(\ve{x}^*)=0$. Then, the Laplace average $f^*_{\lambda_1}(\ve{x})$ corresponds to the limit
\begin{equation}
\label{lim_f_average}
f^*_{\lambda_1}(\ve{x})=\lim_{t \rightarrow \infty} e^{-\sigma_1 t} (f\circ \phi_t)(\ve{x})\,.
\end{equation}
(ii) Complex eigenvalue $\lambda_1$. Consider two observables $f_1 \in C^1$ and $f_2 \in C^1$ that satisfy
\begin{eqnarray*}
&& f_1(\ve{x}^*)=f_2(\ve{x}^*)=0 \\
&& \left|\left \langle \nabla f_1(\ve{x}^*),\ve{a} \right \rangle \right| = \left| \left \langle \nabla f_2(\ve{x}^*),\ve{b} \right \rangle \right| \neq 0\\
&&\left \langle \nabla f_1(\ve{x}^*),\ve{b} \right \rangle = \left \langle \nabla f_2(\ve{x}^*),\ve{a} \right \rangle =0
\end{eqnarray*}
with $\ve{a}=\Re\{\ve{v}_1\}$ and $\ve{b}=-\Im\{\ve{v}_1\}$. Then the  Laplace average $|f^*_{\lambda_1}(\ve{x})|$ of an observable $f\in C^1$ is proportional to the limit
\begin{equation*}
|f^*_{\lambda_1}(\ve{x})| \propto \lim_{\substack{n \rightarrow \infty\\ n\in\mathbb{N}}} e^{-\sigma_1 n T_1}\sqrt{\left( (f_1\circ \phi_{nT_1})(\ve{x})\right)^2+\left((f_2\circ \phi_{nT_1})(\ve{x})\right)^2}\,,
\end{equation*}
with $T_1=2\pi/\omega_1$.
\end{proposition}
\begin{proof}
\emph{(i) Real eigenvalue $\lambda_1$.}
Since $f(\ve{x}^*)=0$, the result follows from \eqref{inequa1} and \eqref{real_average_value}.\\
\emph{(ii) Complex eigenvalue $\lambda_1$.}
Provided that $f(\ve{x}^*)=0$, \eqref{inequa1} implies that
\begin{equation*}
\lim_{n \rightarrow \infty} e^{-\sigma_1 nT_1} (f\circ \phi_{nT_1})(\ve{x})=\left\langle \nabla f(\ve{x}^*), \Re\left\{ \ve{v}_1 e^{i\theta} \right\} \right\rangle e^{\sigma_1 \tau}=\left\langle \nabla f(\ve{x}^*), \ve{a} \cos(\theta) +  \ve{b} \sin(\theta) \right \rangle e^{\sigma_1 \tau}
\end{equation*}
and since $f_1(\ve{x}^*)=f_2(\ve{x}^*)=0$,
\begin{eqnarray*}
\lim_{n \rightarrow \infty} e^{-\sigma_1 nT_1} (f_1\circ \phi_{nT_1})(\ve{x}) & = & \cos(\theta) \left\langle \nabla f_1(\ve{x}^*),\ve{a} \right \rangle e^{\sigma_1 \tau} \\
\lim_{n \rightarrow \infty} e^{-\sigma_1 nT_1} (f_2\circ \phi_{nT_1})(\ve{x}) & = & \sin(\theta) \left\langle \nabla f_2(\ve{x}^*),\ve{b} \right \rangle e^{\sigma_1 \tau}\,.
\end{eqnarray*}
Then, one has
\begin{equation*}
\lim_{n \rightarrow \infty} e^{-\sigma_1 nT_1}\sqrt{\left( (f_1\circ \phi_{nT_1})(\ve{x})\right)^2+\left((f_2\circ \phi_{nT_1})(\ve{x})\right)^2}=|\left \langle \nabla f_1(\ve{x}^*),\ve{a} \right \rangle|  e^{\sigma_1\tau}
\end{equation*}
and it follows from \eqref{abs_average_complex} that the limit is proportional to $|f^*_{\lambda_1}(\ve{x}^*)|$---with the factor of proportionality $2|\left \langle \nabla f_1(\ve{x}^*),\ve{a} \right \rangle/\left \langle \nabla f(\ve{x}^*),\ve{v}_1 \right\rangle|$.
\end{proof}
Proposition \ref{prop_limit} implies that the isostables can be computed as the level sets of particular limits. In the case $\lambda_1\in \mathbb{R}$, the computation of the limit \eqref{lim_f_average} is interpreted as the infinite-dimensional version of the power iteration method used to compute the eigenvector of a matrix associated with the largest eigenvalue. While the straightforward computation of the Laplace averages \eqref{Laplace_average} is characterized by a rate of convergence $T^{-1}$, the computation of these limits is characterized by an exponential rate of convergence. Hence, the results of Proposition \ref{prop_limit} are of great interest from a numerical point of view, and it is particularly so since the numerical instability imposes an upper bound on the finite time horizon $T$.

\begin{remark}
In the case $\lambda_1\in \mathbb{R}$, the limit \eqref{lim_f_average} is characterized by the rate of convergence $\exp(\Re\{\lambda_2-\lambda_1\}T)$, which can still be slow if $\lambda_1\approx \lambda_2$. This rate can be further improved by choosing an observable $f$ that has no projection onto the eigenfunction $s_2$, i.e. that satisfies $\langle \nabla f, \ve{v}_2\rangle=0$. In that case, the rate of convergence will be $\exp(\Re\{\lambda_3-\lambda_1\}T)$. Similarly, the convergence can be made as fast as required by choosing an observable that has no projection onto many other eigenfunctions (i.e. with many zero Koopman modes $\overline{\ve{v}}_{k_1\cdots k_n}$, see Appendix \ref{appendix}).
\end{remark}

\section{Applications}
\label{sec_appli}

The concept of isostables of fixed points is now illustrated with some examples. These examples show that the framework is coherent and general, coherent with the equivalent definition of isostable for excitable systems and general since it is not limited to the particular class of excitable systems.

The isostables are computed according to the algorithm proposed at the beginning of Section \ref{sec_num_comput}. The Laplace averages are numerically computed through the integral \eqref{Laplace_average} (e.g. Section \ref{sec_Lorenz}) or through the limits derived in Proposition \ref{prop_limit} (e.g. Section \ref{sec_FN}).

\subsection{The excitable FitzHugh-Nagumo model}
\label{sec_FN}

The concept of isostables is primarily motivated by the reduction of excitable systems characterized by slow-fast dynamics. In this case, the points on the same isostable $\mathcal{I}_\tau$ share the same asymptotic behavior on a stable slow manifold.

In this context, we compute the isostables for the well-known FitzHugh-Nagumo model \cite{FitzHugh, Nagumo}
\begin{eqnarray*}
\dot{v} & = & -w-v(v-1)(v-a)+I \,,\\
\dot{w} & = & \epsilon (v-\gamma w) \,,
\end{eqnarray*}
which admits an excitable regime with a stable fixed point ($\ve{x}^*=(v^*,w^*)$, with $v^*=w^*$) for the parameters $I=0.05$, $\epsilon=0.08$, $\gamma=1$, and $a=\{0.1,1\}$. The eigenvalues (of the Jacobian matrix at the fixed point) are either real (e.g., $a=1$) or complex (e.g., $a=0.1$). We consider both cases in the sequel.

In \cite{Rabinovitch}, the isostables were computed for the FitzHugh-Nagumo model through the backward integration of trajectories starting in a close neighborhood of the stable slow manifold (or transient attractor). Here, we obtain the same results using a forward integration method based on the computation of the Laplace averages.

\subsubsection{Real eigenvalues ($a=1$)}

The Laplace averages are computed according to the result of Proposition \ref{prop_limit}(i), with the observable $f(v,w)=(v-v^*)+(w-w^*)$. The level sets of the Laplace averages (isostables) are represented in Figure \ref{FN_real}. 

\begin{figure}[h]
\begin{center}
\includegraphics[width=10cm]{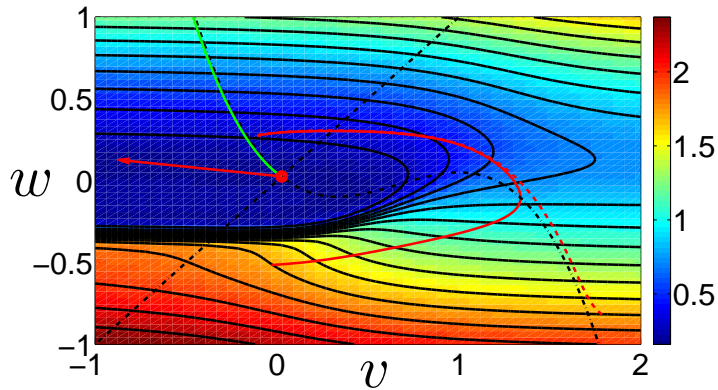}
\caption{The level sets of the Laplace averages $|f^*_{\lambda_1}|$ are the isostables (black curves) of the fixed point (red dot). The color refers to the value of $|f^*_{\lambda_1}|$. In the neighborhood of the fixed point, the isostables are parallel to the direction $\ve{v}_2\approx(-1,0.1133)$ (red arrow). Two trajectories with an initial condition on the same isostable ($(-0.0303,-0.5152)$ for the solid curve, $(1.7879,-0.8182)$ for the dashed curve) synchronously reach the same isostable after a time $\tau-\tau'\approx 12$. They also reach the stable slow manifold (transient attractor) (green curve) synchronously. (The averages are computed on a regular grid $100 \times 100$, with a finite time horizon $T=50$; the black dotted-dashed curves are the nullclines.)}
\label{FN_real}
\end{center}
\end{figure}

One first verifies that the isostables are parallel to the eigenvector $\ve{v}_2$ in the neighborhood of the fixed point. In addition, two trajectories with an initial condition on the same isostable synchronously converge to the fixed point. For instance, two trajectories that start from the same level set $|s_1(\ve{x}')|=1.74$ synchronously reach the level set $|s_1(\ve{x})|=0.17$ after a time $\tau-\tau'\approx 12$. This observation confirms the result of Proposition 1, since
\begin{equation*}
\frac{1}{\sigma_1} \ln \left|\frac{s_1(\ve{x})}{s_1(\ve{x}')}\right|=\frac{1}{-0.1933}\ln \frac{0.17}{1.74}\approx 12\,.
\end{equation*}

The system admits an unstable slow manifold (transient repeller), which corresponds to a stable slow manifold (transient attractor) for the backward-time system. The unstable slow manifold lies in the highly sensitive region $v<0$, $w\approx -0.3$ characterized by a high concentration of isostables. Consider a trajectory that is near the fixed point and that belongs to the isostable $\mathcal{I}_\tau$. If it is weakly perturbed, it will jump to the isostable $\mathcal{I}_{\tau'}$, with $\tau' \approx \tau$, and will reach the initial isostable after a short time $\tau-\tau'\ll 1$. In contrast, if the trajectory is perturbed beyond the unstable slow manifold, it will reach the isostable $\mathcal{I}_{\tau'}$, with $\tau' \ll \tau$. As a consequence, the trajectory will not immediately converge toward its initial position near the fixed point but will exhibit a large excursion in the state space, whose duration is given by $\tau-\tau' \gg 1$. This phenomenon induced by the unstable slow manifold is characteristic of slow-fast excitable systems and is related to the concentration of isostables. Note that for slow-fast asymptotically periodic systems, a high concentration of isochrons is also observed near the unstable slow manifold \cite{Osinga}.

\subsubsection{Complex eigenvalues ($a=0.1$)}
\label{sub_sec_FN_complex}

The Laplace averages are computed according to the result of Proposition \ref{prop_limit}(ii), with the observables $f_1(v,w)=b_2(v-v^*)-b_1(w-w^*)$ and $f_2(v,w)=a_2(v-v^*)-a_1(w-w^*)$, $\ve{a}=(a_1,a_2)$, $\ve{b}=(b_1,b_2)$. The level sets (isostables) are represented in Figure \ref{FN_complex}. We verify that the isostables are ellipses in the neighborhood of the fixed point (Figure \ref{FN_complex}(b)). In addition, two trajectories with an initial condition on the same isostable synchronously converge to the fixed point (Figure \ref{FN_complex}(a)). For instance, two trajectories that start from the same level set $|s_1(\ve{x}')|=0.10$ synchronously reach the level set $|s_1(\ve{x})|=0.051$ after a time $\tau-\tau'\approx 16$. This observation confirms the result of Proposition 1, since
\begin{equation*}
\frac{1}{\sigma_1} \ln \left|\frac{s_1(\ve{x})}{s_1(\ve{x}')}\right|=\frac{1}{-0.041}\ln \frac{0.051}{0.10}\approx 16\,.
\end{equation*}

As in the case $\lambda_1$ real, the system admits an unstable slow manifold (region $v<0$ and $w\approx 0$) characterized by a high concentration of isostables.
\begin{figure}[h]
\subfigure[]{\includegraphics[width=8cm]{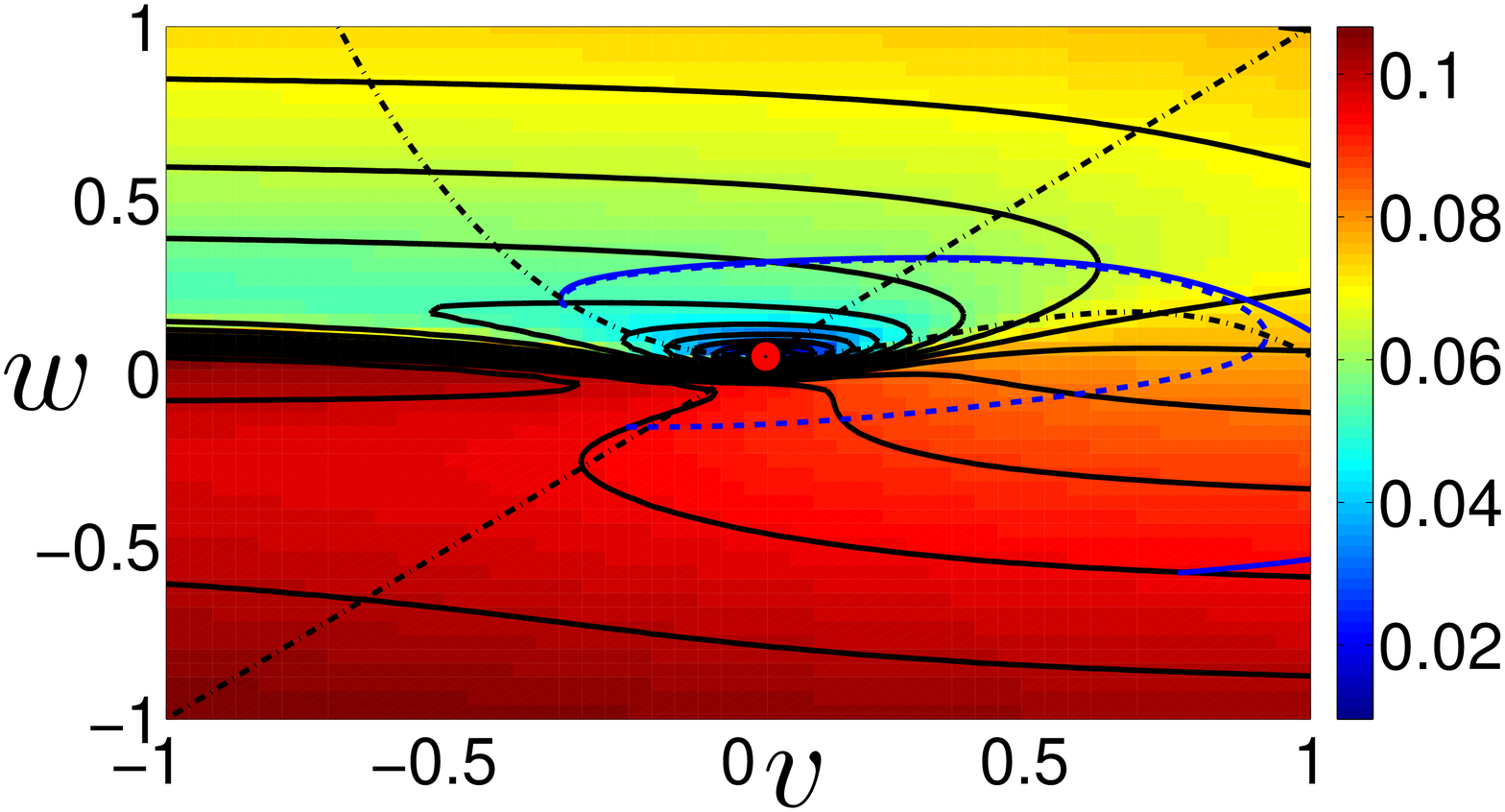}}
\subfigure[]{\includegraphics[width=8cm]{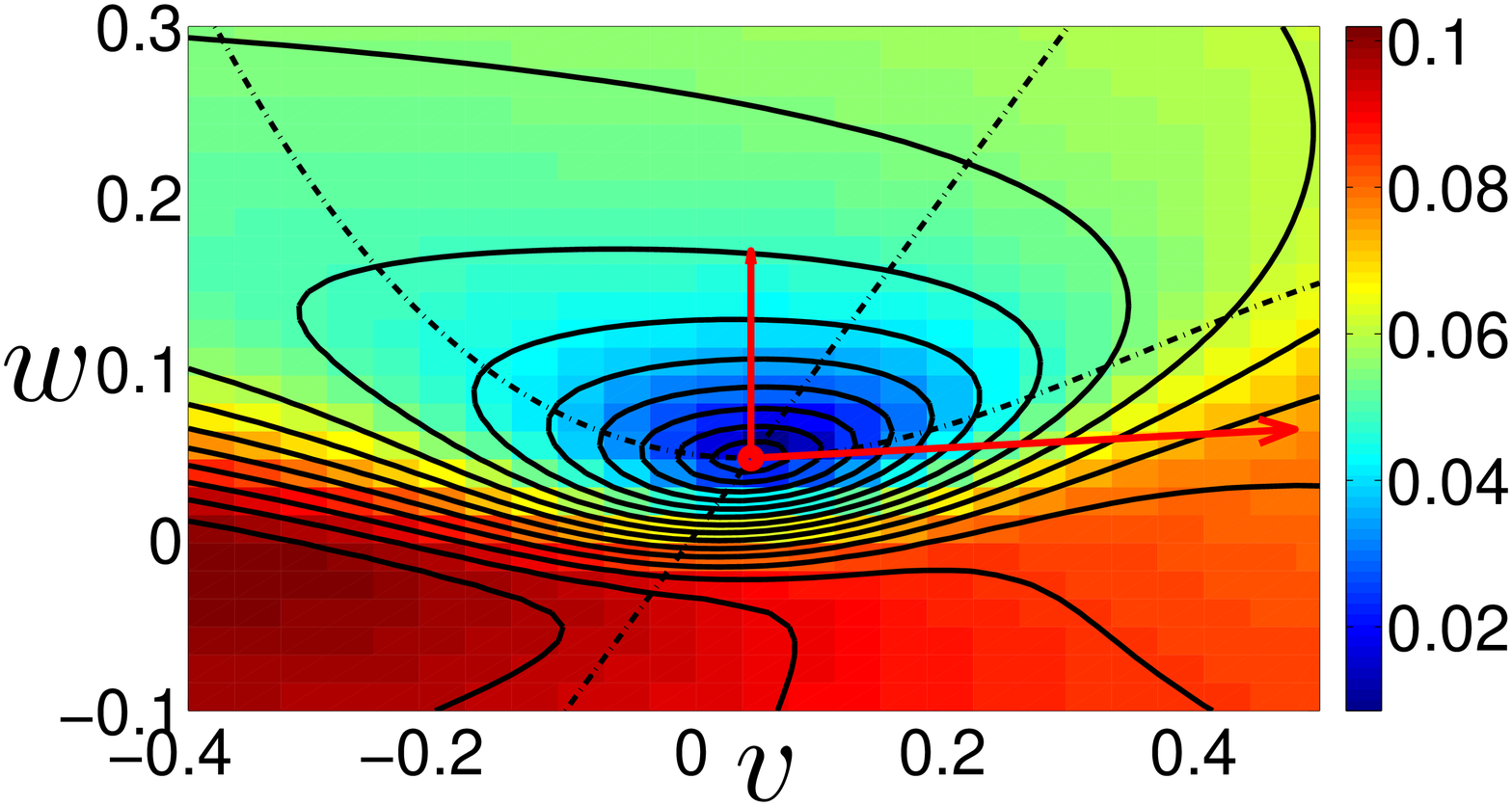}}
\caption{The level sets of the Laplace averages $|f^*_{\lambda_1}|$ are the isostables (black curves) of the fixed point (red dot). (a) Two trajectories with an initial condition on the same isostable ($(0.7688,-0.5779)$ for the solid curve, $(-0.1960,-0.1558)$ for the dashed curve) synchronously reach the same isostable after a time $\tau-\tau'\approx16$. (The averages are computed on a regular grid $100 \times 100$, with a finite time horizon $T=250$, that is, with $11$ iterations of the time-$T_1$ map; the black dotted-dashed curves are the nullclines.) (b) In the neighborhood of the fixed point, the isostables are ellipses. The arrows represent the vectors $\ve{a}=\Re\{\ve{v}_1\}\approx (0.96,0.03)$ and $\ve{b}=-\Im\{\ve{v}_1\}\approx (0,0.27)$. (The averages are computed on a regular grid $50 \times 50$).}
\label{FN_complex}
\end{figure}

\newpage
\subsection{The Lorenz model}
\label{sec_Lorenz}

The framework developed in this paper is not limited to two-dimensional excitable models, but can also be applied to higher-dimensional models, including those which are not characterized by slow-fast dynamics. For instance, we compute in this section the isostables of the Lorenz model
\begin{eqnarray*}
\dot{x_1} & = & a(x_2-x_1) \,,\\
\dot{x_2} & = & x_1(\rho-x_3)-x_2 \,,\\
\dot{x_3} & = & x_1 x_2-b x_3 \,.
\end{eqnarray*}
With the parameters $a=10$, $\rho=0.5$, $b=8/3$, the origin is a stable fixed point with a real eigenvalue $\lambda_1$. Several isostables are depicted in Figure \ref{Lorenz_isostable}. They are the two-dimensional level sets---i.e., the isosurfaces---of the Laplace averages $f^*_{\lambda_1}$ computed for the observable $f(x_1,x_2,x_3)=x_1+x_2+x_3$. Note that the isostables are approximated by a plane in the vicinity of the fixed point.

\begin{figure}[h]
\begin{center}
\includegraphics[width=10cm]{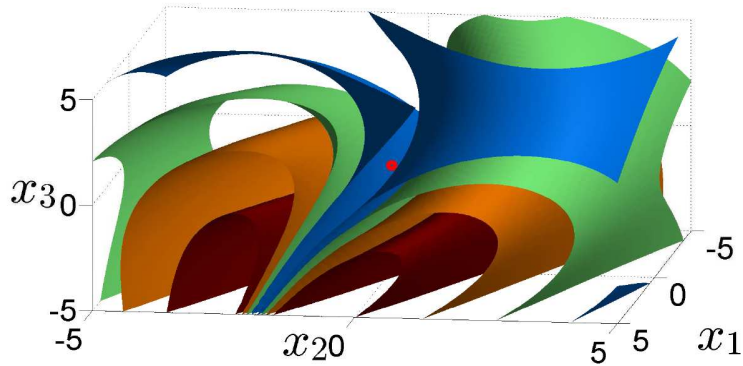}
\caption{The isostables can be computed for three-dimensional models, including those which are not characterized by slow-fast dynamics (in this case, the Lorenz model). Four isostables are represented, which are the level sets of the Laplace averages $|f^*_{\lambda_1}|\in\{0.5, 1, 1.5, 2\}$. (The averages are computed on a regular grid $75 \times 75 \times 75$, with a finite time horizon $T=20$; the red dot corresponds to the fixed point.)}
\label{Lorenz_isostable}
\end{center}
\end{figure}

When the parameter $\rho$ exceeds the critical value $\rho=1$, the origin becomes unstable and two stable fixed points $(\pm x_1^*,\pm x_2^*,x_3^*)$ appear. Since these fixed points are characterized by the same eigenvalues, their isostables can be obtained through the computation of a single Laplace average $f^*_{\lambda_1}$. In Figure \ref{Lorenz_isostable_complex}, the isostables are computed for the value $\rho=2$, a situation characterized by a complex eigenvalue $\lambda_1$. Note that the isostables are cylinders in the vicinity of the fixed point. In addition, the level set $|f^*_{\lambda_1}|\rightarrow \infty$ corresponds to the separatrix between the two basins of attraction (i.e. the stable manifold of the fixed point at the origin).

\begin{figure}[h]
\begin{center}
\includegraphics[width=10cm]{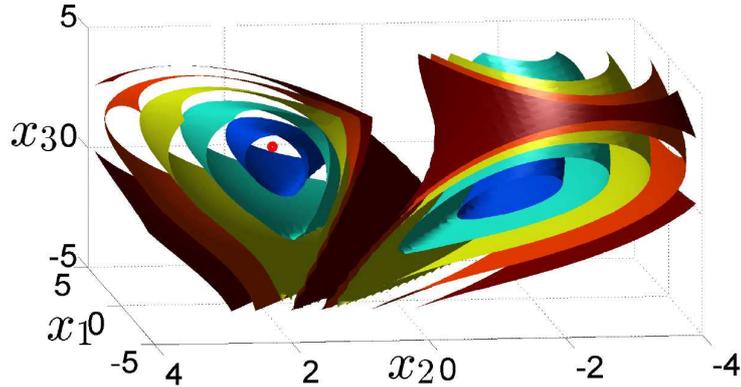}
\caption{The level sets of the Laplace averages $|f^*_{\lambda_1}|\in\{1,2,3,4,5\}$ represent five isostables of the two stable fixed points. (The averages are computed on a regular grid $50 \times 50 \times 50$, with a finite time horizon $T=15$; the red dot corresponds to the (visible) stable fixed point.)}
\label{Lorenz_isostable_complex}
\end{center}
\end{figure}

\section{Discussion}
\label{sec_Lyap_lin}

In this section, we discuss some topics related to the concept of isostables. Through the Koopman operator framework, we claim that the notion of isostables is different from but complementary to the known notion of isochrons. Isostables and isochrons define a set of action-angle coordinates and are related to a global linearization of the dynamics. In addition, we briefly show that the isostables are the level sets of a particular Lyapunov function for the fixed point dynamics.

\subsection{Isostables vs. isochrons}
\label{subsec_isochrons}

The isostables are the sets of points that approach the same trajectory when they converge toward the fixed point. Similarly, in the case of asymptotically periodic systems, the isochrons are the set of points that converge toward the same trajectory on the limit cycle \cite{Winfree2}. It follows that isostables (of fixed points) and isochrons (of limit cycles) are conceptually related. However, these two concepts are also characterized by intrinsic differences and turn out to be complementary.

The difference between isostables and isochrons can be understood through the framework of the Koopman operator. The isostables have been defined as the level sets of the \emph{absolute value} of the Koopman eigenfunction $|s_1(\ve{x})|$ (Section \ref{iso_nonlinear}). In contrast, the isochrons of limit cycles were computed in \cite{Mauroy_Mezic} by using the \emph{argument} of a Koopman eigenfunction. Similarly, the isochrons of fixed points (characterized by a complex eigenvalue $\lambda_1$) can be defined as the levels sets of the argument $\angle s_1(\ve{x})$. These sets (also called isochronous sections) are well-known and usually defined as the sets invariant under a particular return map (i.e. the discrete map $\phi(T_1,\cdot)$ considered in Proposition \ref{prop_limit}). Also, their existence, which is not trivial in the case of weak foci (i.e. purely imaginary eigenvalues) or nonsmooth vector fields, has been investigated in \cite{Gine, Sabatini}. In the case of linear systems, the isochrons correspond to radial lines that intersect at the fixed point (see Figure \ref{linear_complex}(b)). For nonlinear systems, they are tangent to radial lines at the fixed point but are characterized by a more complex geometry (see Figure \ref{FN_isochrons}). Note that, when they exist, the isochrons are uniquely determined by their toplogical properties: they define the unique periodic partition of the state space (of period $T_1$). In contrast, more care was needed to define the isostables as the level sets of the unique smooth eigenvalue $s_1$.

Isostables and isochrons appear to be two different but complementary notions. On one hand, the isostables are related to the \emph{stability} property of the system and provide information on how fast the trajectories converge \emph{toward the attractor}. On the other hand, the isochrons are related to a notion of \emph{phase} and provide information on the asymptotic behavior of the trajectories \emph{on the attractor}. Given \eqref{prop_eigenfunction}, the isostables are related to the property
\begin{equation}
\label{prop_isostable}
\frac{d}{dt}|s_1(\phi_t(\ve{x}))|=\sigma_1 |s_1(\phi_t(\ve{x}))|
\end{equation}
while the isochrons are characterized by
\begin{equation}
\label{prop_isochron}
\frac{d}{dt}\angle s_1(\phi_t(\ve{x}))=\omega_1 \,.
\end{equation}
In the case of fixed points, it is clear that the isochrons are not relevant to characterize the synchronous convergence of the trajectories, a fact that stresses the importance of considering the isostables instead.

\begin{figure}[h]
\begin{center}
\includegraphics[width=10cm]{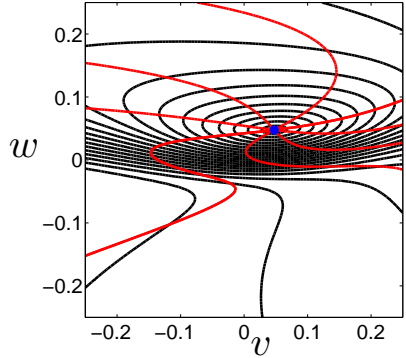}
\caption{For a fixed point with a complex eigenvalue $\lambda_1$, the isostables (black curves) and the isochrons (red curves) of the fixed point are the level sets of $|s_1(\ve{x})|$ and $\angle s_1(\ve{x})$, respectively. In the vicinity of the fixed point, the isostables are ellipses and the isochrons are straight lines. (The numerical computations are performed for the FitzHugh-Nagumo model, with the parameters considered in Section \ref{sub_sec_FN_complex}; the blue dot represents the fixed point.)}
\label{FN_isochrons}
\end{center}
\end{figure}

\subsection{Action-angle coordinates and global linearization}
\label{subsec_linearization}

For a two-dimensional dynamical system which admits a spiral sink (two complex eigenvalues), the families of isostables and isochrons provide an action-angle coordinates representation of the dynamics. More precisely, \eqref{prop_isostable} and \eqref{prop_isochron} imply that, with the variables $(r,\theta)=(|s_1(\ve{x})|,\angle s_1(\ve{x}))$, the system is characterized by the (action-angle) dynamics
\begin{eqnarray*}
\dot{r} & = & \sigma_1 r\\
\dot{\theta} & = & \omega_1
\end{eqnarray*}
in the basin of attraction of the fixed point. For systems of higher dimension, the action-angle dynamics are obtained with several Koopman eigenfunctions, i.e. $(r_j,\theta_j)=(|s_j(\ve{x})|,\angle s_j(\ve{x}))$ leads to $\dot{r}_j=\sigma_j r_j$, $\dot{\theta}_j=\omega_j$. Note that this was also shown in Section \ref{subsub_complex} in the case of linear systems with a spiral sink.

When expressed in the action-angle coordinates, the dynamics become linear. This is in agreement with the recent work \cite{Lan} showing that a coordinate system which linearizes the dynamics is naturally provided by the eigenfunctions of the Koopman operator (see also Appendix \ref{appendix}). Namely, in the new variables $y_j=s_j(\ve{x})$, the system dynamics are given by
\begin{equation*}
\frac{d}{dt}\left(\begin{array}{c} y_1\\ \vdots \\ y_n \end{array}\right) = \left(\begin{array}{ccc} \lambda_1 & & 0 \\ & \ddots & \\ 0 & & \lambda_2 \end{array}\right) \left(\begin{array}{c} y_1 \\ \vdots \\ y_n \end{array}\right)\,.
\end{equation*}
Moreover, the linear change of coordinates
\begin{equation}
\label{z_var}
\left(\begin{array}{c} z_1\\ \vdots \\ z_n \end{array}\right)=\ve{V} \left(\begin{array}{c} y_1\\ \vdots \\ y_n \end{array}\right)\,,
\end{equation}
where the columns of $\ve{V}$ are the eigenvectors $\ve{v}_j$ of the Jacobian matrix $\ve{J}$ at the fixed point, leads to the linear dynamics
\begin{equation*}
\frac{d}{dt}\left(\begin{array}{c} z_1 \\ \vdots \\ z_n \end{array}\right) = \ve{J} \left(\begin{array}{c} z_1 \\ \vdots \\ z_2 \end{array}\right) \,.
\end{equation*}
For the two-dimensional FitzHugh-Nagumo model, the coordinates $(z_1,z_2)$ are represented in Figure \ref{FN_linearized} and are equivalent to the action-angle coordinates $(r,\theta)$ (Figure \ref{FN_isochrons}). They correspond to Cartesian coordinates in the vicinity of the fixed point, where the linearized dynamics are a good approximation of the nonlinear dynamics (see also \eqref{approx_z} in Appendix \ref{appendix}). But owing to the nonlinearity, the coordinates are deformed as their distance from the fixed point increases. The comparison between these coordinates and regular Cartesian coordinates therefore appears as a \emph{measure of the system nonlinearity}.

In the case of two-dimensional systems with a stable spiral sink, the derivation of action-angle coordinates and the global linearization are obtained through the isostables and the isochrons, that is, with only the first Koopman eigenfunction $s_1(\ve{x})$. For higher-dimensional systems (or two-dimensional systems with a sink node), global linearization involves several Koopman eigenfunctions $s_j(\ve{x})$ (see \cite{Lan} for a detailed study), which can be obtained through the generalized Laplace averages (see Remark \ref{gen_Laplace_av}). In the context of model reduction, or when the dynamics are significantly slow in one particular direction, the first eigenfunction---related to the isostable---is however sufficient to retain the main information on the system behavior.

\begin{figure}[h]
\begin{center}
\includegraphics[width=10cm]{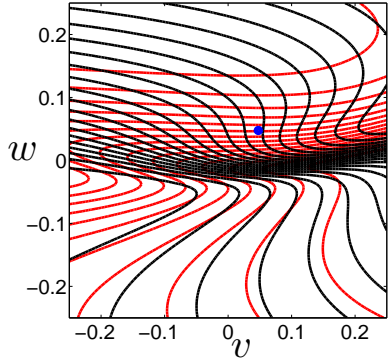}
\caption{The coordinates $z_1$ (black curves) and $z_2$ (red curves) correspond to Cartesian coordinates in the vicinity of the fixed point but are deformed when far from the fixed point. (The numerical computations are performed for the FitzHugh-Nagumo model, with the parameters considered in Section \ref{sub_sec_FN_complex}; the blue dot represents the fixed point.)}
\label{FN_linearized}
\end{center}
\end{figure}


\subsection{Lyapunov function and contracting metric}

As a consequence of the linearization properties illustrated in the previous section, the Koopman eigenfunctions---and in particular the isostables---can be used to derive Lyapunov functions and contracting metrics for the system.

In the particular case of two-dimensional systems with a spiral sink, the isostables are the level sets of the particular Lyapunov function $\mathcal{V}(\ve{x})=|s_1(\ve{x})|$ (see Figure \ref{FN_Lyapunov} for the FitzHugh-Nagumo model). Indeed, \eqref{prop_isostable} implies that $\dot{\mathcal{V}}(\ve{x})=\sigma_1 \mathcal{V}(\ve{x})<0$  $\forall \ve{x}\in \mathcal{B}(\ve{x}^*)\setminus \{\ve{x}^*\}$ and one verifies that $\mathcal{V}(\ve{x}^*)=0$. This function is a special Lyapunov function of the system, in the sense that its decay rate is constant everywhere. (Note that the function $\mathcal{V}=\ln(|s_1(\ve{x})|)/\sigma_1$ satisfies $\dot{\mathcal{V}}=-1$ but with $\mathcal{V}(\ve{x}^*)=-\infty$.)
\begin{figure}[h]
\begin{center}
\includegraphics[width=12cm]{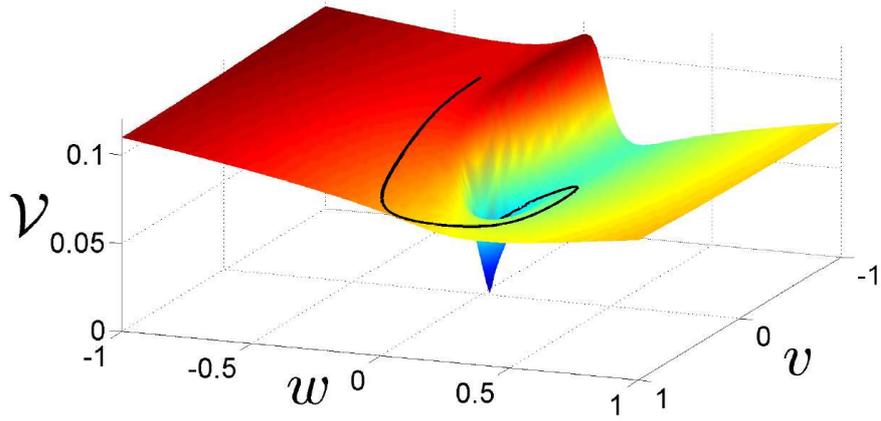}
\caption{The function $\mathcal{V}=|s_1(\ve{x})|$ is a particular Lyapunov function for the system (here, the FitzHugh-Nagumo model with the parameters considered in Section \ref{sub_sec_FN_complex}). One verifies that the function decreases with a constant rate along a trajectory (black curve). Note also that the unstable slow manifold (region $v<0$, $w\approx0$) is characterized by a line of maxima of the Lyapunov function.}
\label{FN_Lyapunov}
\end{center}
\end{figure}

In addition, the isostables are related to a metric which is contracting in the basin of attraction of the fixed point. Namely, the distance
\begin{equation*}
d(\ve{x},\ve{x}')=|s_1(\ve{x})-s_1(\ve{x}')|
\end{equation*}
is well-defined and \eqref{prop_eigenfunction} implies that
\begin{equation*}
\frac{d}{dt} d\big(\phi_t(\ve{x}),\phi_t(\ve{x}')\big)=\sigma_1 d(\ve{x},\ve{x}')<0  \,, \quad \forall  \ve{x} \neq \ve{x}' \in \mathcal{B}(\ve{x}^*)\,.
\end{equation*}
For more general systems that admit a stable fixed point, the function $\mathcal{V}(\ve{x})=|s_1(\ve{x})|$ is still decreasing along the trajectories, but $\mathcal{V}(\ve{x})=0$ does not imply $\ve{x}=\ve{x}^*$ ($\mathcal{V}$ is zero on the whole isostable $\mathcal{I}_{\tau=\infty}$ that contains the fixed point). However, the function can be used with the LaSalle invariance principle. To obtain a good Lyapunov function, several Koopman eigenfunctions must be considered. For instance, the function
\begin{equation*}
\mathcal{V}(\ve{x})=\left(\sum_{j=1}^n |s_j(\ve{x})|^p \right)^{1/p}\,,
\end{equation*}
with the integer $p\geq 1$, satisfies
\begin{equation*}
\dot{\mathcal{V}}(\ve{x})=\left(\sum_{j=1}^n |s_j(\ve{x})|^p \right)^{\frac{1}{p}-1} \sum_{j=1}^n \sigma_j |s_j(\ve{x})|^p \leq \sigma_1 \mathcal{V}(\ve{x})
\end{equation*}
and $\mathcal{V}(\ve{x})=0$ iff $\ve{x}=\ve{x}^*$. In addition, a contracting metric is given by
\begin{equation*}
d\big(\ve{x},\ve{x}'\big)=\left(\sum_{j=1}^n |s_j(\ve{x})-s_j(\ve{x}')|^p\right)^{1/p}
\end{equation*}
and one has
\begin{equation*}
\frac{d}{dt} d\big(\phi_t(\ve{x}),\phi_t(\ve{x}')\big) \leq \sigma_1 d(\ve{x},\ve{x}')\,, \quad \forall  \ve{x}, \ve{x}' \in \mathcal{B}(\ve{x}^*)\,.
\end{equation*}
It follows from the above observations that showing the existence of stable eigenfunctions of the Koopman operator is sufficient to prove the global stability of the attractor. Therefore, the Koopman operator framework could potentially yield an alternative method for the global stability analysis of nonlinear systems.

\section{Conclusion}
\label{sec_conclu}

In this paper, the well-known phase reduction of asymptotically periodic systems has been extended to the class of systems which admit a stable fixed point. In the context of the Koopman operator framework, the approach is not restricted to excitable systems with slow-fast dynamics but is valid in more general situations. The isostables required for the reduction of the dynamics, which correspond is some cases to the fibers of a particular invariant manifold of the system, are interpreted as the level sets of an eigenfunction of the Koopman operator. In addition, they are shown to be different from the concept of isochrons that prevails for asymptotically periodic systems. Beyond its theoretical implications, the framework also yields an efficient (forward integration) method for computing the isostables, which is based on the estimation of Laplace averages along the trajectories.

The reduction of the dynamics through the Koopman operator framework leads to an action-angle coordinates representation that is intimately related to a global linearization of the system. More precisely, the proposed reduction procedure is nothing but a global linearization of the system where only one direction of interest is considered, which retains the main information on the system behavior (i.e. the slowest direction). In this context, the isostables---related to the action---or the isochrons---related to the angle--- used for the reduction are particular objects involved in the global linearization process. Given this relation between reduction methods and linearization, research perspectives are twofold. On the one hand, convenient Laplace average methods could be developed for linearization purposes (e.g. computation of the isostables of limit cycles \cite{Guillamon} in the whole---possibly high-dimensional---basin of attraction), and for the computation of (un)stable manifolds as well. On the other hand, the Koopman operator framework can be further exploited for the reduction of more general dynamical systems (e.g. chaotic systems).

\section*{Acknowledgments}
The work was completed while A. Mauroy held a postdoctoral fellowship from the Belgian American Educational Foundation and was partially funded by Army Research Office Grant W911NF-11-1-0511, with Program Manager Dr. Sam Stanton.

\appendix

\section{Spectral decomposition of the Koopman operator}
\label{appendix}

In this appendix, we derive the expansion \eqref{evol_observable} of an observable onto the eigenfunctions of the Koopman operator. Consider the change of variable $\ve{s}:\ve{x} \mapsto \ve{y}$, with $y_j=s_j(\ve{x})$, where $s_j$ is an eigenfunction of the Koopman operator. It follows that $\ve{s}(\ve{x}^*)=\ve{0}$ and, given \eqref{prop_eigenfunction}, the dynamics is linearized in the $\ve{y}$ variable, i.e. $\dot{y}_j=\lambda_j \, y_j$. According to the linearization Poincaré theorem \cite{Gaspard}, the transformation $\ve{s}$ is analytic since the vector field $F$ is analytic and the eigenvalues are nonresonant (and provided there is no unstable fixed point in $\mathcal{B}(\ve{x}^*)$). If an observable $f$ is analytic, the Taylor expansion of $f(\ve{s}^{-1}(\ve{y}))$ around the origin yields
\begin{equation}
\label{Taylor_expansion}
f(\ve{s}^{-1}(\ve{y}))=f(\ve{x}^*)+\nabla f^T(\ve{x}^*) \, \ve{J}_{\ve{s}^{-1}} \ve{y} + \frac{1}{2} \ve{y}^T\, \ve{J}^T_{\ve{s}^{-1}} \ve{H} \ve{J}_{\ve{s}^{-1}} \, \ve{y} + \frac{1}{2} \ve{y}^T \sum_{k=1}^n \left.\frac{\partial f}{\partial x_k}\right|_{\ve{x}^*} \ve{H}_{s^{-1}_k} \ve{y} + \textrm{h.o.t.}\,,
\end{equation}
where $\ve{J}_{\ve{s}^{-1}}$ is the Jacobian matrix of $\ve{s}^{-1}$ at the origin (i.e. $J_{\ve{s}^{-1},ij}=\partial s_i^{-1}/\partial y_j(\ve{0})$), $\ve{H}$ is the Hessian matrix of $f$ at $\ve{x}^*$ (i.e. $H_{ij}=\partial^2 f/(\partial x_i \partial x_j)(\ve{x}^*)$), and $\ve{H}_{s^{-1}_k}$ is the Hessian matrix of $s_k^{-1}$ at the origin (i.e. $H_{s^{-1}_k,ij}=\partial^2 s^{-1}_k/(\partial y_i \partial y_j)(\ve{0})$). Using the relationship $\ve{y}=(s_1(\ve{x}),\dots,s_n(\ve{x}))$, we can turn the expansion \eqref{Taylor_expansion} into an expansion of $f$ onto the products of the eigenfunctions $s_j$. For a vector-valued observable $\ve{f}$, we obtain
\begin{equation}
\label{equa_expansion_f}
\ve{f}(\ve{x})=\sum_{\{k_1,\dots,k_n\}\in\mathbb{N}^n} \ve{\overline{v}}_{k_1\cdots k_n} \, s_1^{k_1}(\ve{x}) \cdots s_n^{k_n}(\ve{x})
\end{equation}
with the (first) Koopman modes
\begin{equation*}
\overline{\ve{v}}_{k_1\cdots k_n}=\begin{cases}
\ve{f}(\ve{x}^*) & k_j=0 \, \forall j\,, \\
\displaystyle
\sum_{k=1}^n \left.\frac{\partial \ve{f}}{\partial x_k}\right|_{\ve{x}^*} \left.\frac{\partial s_k^{-1}}{\partial y_j}\right|_{\ve{0}} & k_j=1\,, \,\, k_i=0 \, \forall i \neq j\,, \\
\displaystyle
\sum_{k=1}^n \sum_{l=1}^n \left.\frac{\partial^2 \ve{f}}{\partial x_k \partial x_l}\right|_{\ve{x}^*} \left.\frac{\partial s_k^{-1}}{\partial y_i}\right|_{\ve{0}} \left.\frac{\partial s_l^{-1}}{\partial y_j}\right|_{\ve{0}} + \sum_{k=1}^n \left.\frac{\partial \ve{f}}{\partial x_k}\right|_{\ve{x}^*} \left.\frac{\partial^2 s_k^{-1}}{\partial y_i \partial y_j}\right|_{\ve{0}} & k_i=k_j=1\,, \,\, k_r=0 \, \forall r\neq \{i,j\} \,,\\
\displaystyle
\frac{1}{2}\sum_{k=1}^n \sum_{l=1}^n \left.\frac{\partial^2 \ve{f}}{\partial x_k \partial x_l}\right|_{\ve{x}^*} \left.\frac{\partial s_k^{-1}}{\partial y_i}\right|_{\ve{0}} \left.\frac{\partial s_l^{-1}}{\partial y_i}\right|_{\ve{0}} + \frac{1}{2} \sum_{k=1}^n \left.\frac{\partial \ve{f}}{\partial x_k}\right|_{\ve{x}^*} \left.\frac{\partial^2 s_k^{-1}}{\partial y_i^2} \right|_{\ve{0}} & k_i=2\,, \,\, k_j=0 \, \forall j\neq i \,.
\end{cases}
\end{equation*}
The other (higher-order) Koopman modes can be derived similarly from \eqref{Taylor_expansion}. Since the eigenfunctions satisfy \eqref{prop_eigenfunction}, the relationship \eqref{evol_observable} directly follows from \eqref{equa_expansion_f}.

For the observable $f(\ve{x})=\ve{x}$, the Koopman modes are given by
\begin{equation*}
\ve{v}_{k_1\cdots k_n} = \frac{1}{k_1 ! \dots k_n !} \left.\frac{\partial^{k_1\cdots k_n} \ve{s}^{-1}}{\partial^{k_1} y_1 \cdots \partial^{k_n} y_n} \right |_{\ve{0}} \,.
\end{equation*}
In particular, the eigenvectors of the Jacobian matrix $\ve{J}$ of $F$ (i.e. $\ve{v}_j=\ve{v}_{k_1\cdots k_n}$, with $k_j=1$, $k_i=0 \, \forall i \neq j$) correspond to
\begin{equation*}
\ve{v}_j=\left. \frac{\partial \ve{s}^{-1}}{\partial y_j}\right|_{\ve{0}}
\end{equation*}
and one has $\ve{J}_{\ve{s}^{-1}}=\ve{V}$, where the columns of $\ve{V}$ are the eigenvectors $\ve{v}_j$. It follows that the variables $\ve{z}$ introduced in \eqref{z_var} satisfy $\ve{z}=\ve{J}_{\ve{s}^{-1}} \ve{y}$ so that \eqref{Taylor_expansion} implies \begin{equation}
\label{approx_z}\ve{x}=\ve{x}^*+\ve{z}+o(\|\ve{z}\|)\,.
\end{equation}

In addition, the derivation of $\ve{y}=\ve{s}(\ve{s}^{-1}(\ve{y}))$ at the origin leads to
\begin{equation*}
\delta_{ij} = \left \langle \nabla s_i(\ve{x}^*) , \left.\frac{\partial \ve{s}^{-1}}{\partial y_j}\right|^c_{\ve{0}} \right \rangle = \left \langle \nabla s_i(\ve{x}^*) , \ve{v}^c_j \right \rangle\,.
\end{equation*}
Therefore, the gradient $\nabla s_i(\ve{x}^*)$ is the left eigenvector $\tilde{\ve{v}}^c_i$ of $\ve{J}$ (associated with the eigenvalue $\lambda_i$) and one has
\begin{equation}
\label{approx_eigen}
s_i(\ve{x}) = \langle \ve{x}-\ve{x}^*, \nabla s_i^c (\ve{x}^*) \rangle + o(\|\ve{x}-\ve{x}^*\|)= \langle \ve{x}-\ve{x}^*,\tilde{\ve{v}}_i \rangle +o(\|\ve{x}-\ve{x}^*\|) \,,
\end{equation}
which implies that, for $\|\ve{x}-\ve{x}^*\|\ll 1$, the eigenfunction $s_i(\ve{x})$ is well approximated by the eigenfunction of the linearized system.

\bibliographystyle{siam}

\end{document}